\documentclass{amsart}
\usepackage[T1]{fontenc}        
\usepackage[utf8]{inputenc}     
\usepackage[english]{babel}
\usepackage{layout}

\usepackage[dvipsnames]{xcolor}

\usepackage{lmodern}            % Latin Modern
\usepackage{microtype}          % améliore la lisibilité du texte
\usepackage{mathptmx}           % Times pour texte et maths, un peu plus gras

\usepackage{amsmath,amssymb,amsfonts,amsthm}
\usepackage{mathrsfs}

\usepackage{graphicx}
\usepackage{layout}

\usepackage[dvipsnames]{xcolor}

\usepackage{amssymb}
\usepackage{fancybox}
\usepackage{fancyhdr}
\usepackage{times}

\usepackage{verbatim}
\usepackage[pdftex,bookmarks,colorlinks,linkcolor=blue]{hyperref}

\usepackage{setspace}
\usepackage{lmodern}
\usepackage{colortbl}

\numberwithin{equation}{section}
\newcommand{\N}{\mathbb{N}}
\newcommand{\R}{\mathbb{R}}
\newcommand{\C}{\mathbb{C}}

\newcommand{\ccA}{{\widetilde{\mathcal{A}}}}

\newcommand{\cF}{{\mathcal{F}}}

\newcommand{\cC}{{\mathcal{C}}}

\newcommand{\cD}{{\mathcal{D}}}

\newcommand{\cK}{{\mathcal{K}}}

\newcommand{\dist}{\mathrm{dist}}

\newcommand{\beqa}{\begin{eqnarray*}}
\newcommand{\bea}{\begin{eqnarray}}
\newcommand{\eeqa}{\end{eqnarray*}}
\newcommand{\eea}{\end{eqnarray}}

\newtheorem{lem}{Lemma}[section]

\newtheorem{cor}{Corollary}[section]

\newtheorem{theo}{Theorem}[section]

\title[Complete interpolating sequences for  Fock type spaces]{Complete interpolating sequences for Fock type spaces}

\author[]{Karim Kellay}
\address{{K. Kellay:} Univ. Bordeaux, CNRS, Bordeaux INP, IMB, UMR 5251, F-33400 Talence, France}
\email{kkellay@math.u-bordeaux.fr}

\author[]{Youssef Omari}
\address{{Y. Omari:} LAGA, Faculty of Sciences, Ibn Tofail University, Kenitra, Morocco}
\email{youssef.omari@uit.ac.ma \& omariysf@gmail.com}

\subjclass{30H20, 30H05, 46E22, 42C15}
\keywords{weighted Fock spaces, reproducing kernels, uniqueness sets, sampling and interpolating sequences, complete interpolating sequences}
\begin{document}
\maketitle
\begin{abstract}
We obtain a characterization of complete interpolating sequences in a class of Fock-type spaces with radial weights for which such sequences exist. 
Our criterion is formulated in terms of logarithmic separation and controlled perturbations of a reference sequence satisfying an Avdonin-type condition. 
This provides a geometric description of complete interpolating sequences and extends previous results of Borichev--Lyubarskii and Baranov--Belov--Borichev on Riesz bases of reproducing kernels in Fock-type spaces. 
It also yields explicit density criteria for sampling and interpolating sequences.
\end{abstract}
\section{Introduction}

Sampling and interpolation problems play a central role in the theory of spaces of analytic functions. 
Given a Hilbert space of analytic functions with reproducing kernel, a natural question is to determine which discrete sequences allow one to reconstruct functions from their values on the sequence and which sequences allow arbitrary interpolation. 
Sequences possessing the first property are called \emph{sampling sequences}, while those possessing the second are called \emph{interpolating sequences}. 
Sequences enjoying both properties are known as \emph{complete interpolating sequences}. 
The existence and structure of such sequences reflect a delicate interaction between the geometry of the space and the growth of its reproducing kernels.

A classical example is provided by the Paley--Wiener space of entire functions of exponential type. 
In this setting the Kadets--Ingham theorem describes the stability of exponential Riesz bases under small perturbations of the integer lattice (see \cite{K,Young1980}). 
More generally, in de Branges spaces complete interpolating sequences correspond exactly to Riesz bases of reproducing kernels \cite{deBranges1968}. 
These ideas have strongly influenced the modern theory of interpolation in spaces of analytic functions.

Weighted Fock spaces provide another natural framework in which sampling and interpolation problems arise. 
For the classical Gaussian Fock space, Seip and Wallstén obtained a complete description of sampling and interpolating sequences in terms of density conditions \cite{Seip92,SW92}. 
In particular, although sampling sequences and interpolating sequences both exist in this space, no sequence can possess both properties simultaneously.  Further results in \cite{BS93,BHKM} show that there are no sequences which are simultaneously sampling and interpolating, even when {unbounded} multiplicities are allowed.

A different phenomenon appears in Fock-type spaces associated with weights of slower growth. 
In particular, Borichev and Lyubarskii showed that certain Fock-type spaces admit Riesz bases of reproducing kernels \cite{BL}. 
Later, Baranov, Belov, and Borichev established a deep connection between these spaces and de Branges spaces \cite{BBB17}, providing a structural explanation for the existence of Riesz bases in this setting. More general weights, whose Laplacian is a doubling measure, were considered by Berndtsson and Ortega-Cerdà \cite{BO} and Marco, Massaneda, and Ortega-Cerdà \cite{NMO}. Related questions for rapidly growing radial weights were studied by Lyubarskii and Seip \cite{LS1} and in \cite{BDK06}. In this case, complete interpolating sequences do not exist.   For small Fock spaces with slowly growing radial (logarithmic) weights, complete descriptions of complete interpolating sequences are available; see \cite{BB,BBG,BDHK,BM,KO20,Mi,O21} for various settings and extensions.

The purpose of this paper is to give a complete characterization of complete interpolating sequences in a class of weighted Fock spaces with slowly growing radial weights.
We show that, in this setting, the existence of such sequences is governed by a critical density condition.
In particular, the role of the reference lattice is played by a sequence whose geometry naturally reflects the underlying Fock-space weight. \\

We now introduce the Fock-type spaces and the associated notions of sampling and interpolation that will be used throughout the paper. Let $\varphi:[0,\infty)\rightarrow (0,\infty)$,  and extend $\varphi$ to the complex plane $\C$ by putting $\varphi(z)=\varphi(|z|)$. For every $0<p\leq \infty$ and every weight $\varphi$ we associate the Fock spaces $\cF^p_\varphi$ of all entire functions $f$ satisfying 
\[
\begin{aligned}
\|f\|_{\varphi,\infty} 
&:= \sup_{z\in\mathbb{C}} |f(z)|e^{-\varphi(z)} 
< \infty, 
\quad p=\infty, \\[6pt]
\|f\|_{\varphi,p}^p 
&:= \int_{\mathbb{C}} |f(z)|^p e^{-p\varphi(z)}\, dm(z) 
< \infty, 
\quad 0<p<\infty,
\end{aligned}
\]
{where $dm$ is the} area Lebesgue measure.  We assume that $\log |z| = o(\varphi(z))$ as $|z|\to \infty$, ensuring that the space is not limited to polynomials.

For any $z\in\C$, we consider the linear maps 
$$L_z\ :\ f\in\cF^p_\varphi\longmapsto f(z)\in\C.$$ 
In the case {$p=2$, the boundedness of $L_z$ and the} Riesz representation theorem ensure that there exists a unique function $K_z\in\cF^2_\varphi$, called the reproducing kernel at $z$,  such that 
\[
L_z(f) = \langle f,K_z\rangle,\quad f\in\cF^2_\varphi.
\]
These reproducing kernels are central to the study of Fock-space structure; for sufficiently regular weights $ \varphi(r) \gg (\log r)^2$, \(\cF^2_\varphi\) does not admit a Riesz basis of normalized kernels \cite{BDK06,BL}.  The classical Fock space $\varphi(r) = r^2$ had been treated earlier by Seip \cite{Seip92,SW92}. By contrast, when $\varphi(r) = (\log r)^\alpha$ with $1 < \alpha \le 2$, $\cF^2_\varphi$ admits Riesz bases of normalized reproducing kernels supported on the real line \cite{BL}, making these spaces de Branges spaces.

We now introduce the notions of sampling and interpolation sequences. Let $\Lambda$ be a sequence of $\C$.
For $0<p<\infty$, the sequence $\Lambda$ is called a \emph{sampling sequence} for $\cF^p_\varphi$ if there exist constants $0<c_\Lambda \le C_\Lambda$ such that
\[
c_\Lambda \sum_{\lambda\in \Lambda} 
\frac{|f(\lambda)|^p}{\|L_\lambda\|^p}
\le 
\|f\|_{\varphi,p}^p
\le 
C_\Lambda \sum_{\lambda\in \Lambda} 
\frac{|f(\lambda)|^p}{\|L_\lambda\|^p},
\quad f\in\cF^p_\varphi.
\]
The sequence $\Lambda$ is called an \emph{interpolating sequence} for $\cF^p_\varphi$ if for every sequence 
$v=(v_\lambda)_{\lambda\in\Lambda}$ such that 
$\left(v_\lambda/\|L_\lambda\|\right)_{\lambda\in\Lambda}\in \ell^p$,
there exists $f\in\cF^p_\varphi$ satisfying
\[
f(\lambda)=v_\lambda, \quad \lambda\in\Lambda .
\]
{For $p=\infty$}, the sequence $\Lambda$ is called \emph{sampling} for $\cF^\infty_\varphi$ if there exists $C_\Lambda\ge1$ such that
\[
\|f\|_{\varphi,\infty}
\le
C_\Lambda \sup_{\lambda\in\Lambda}
|f(\lambda)|e^{-\varphi(\lambda)},
\quad f\in\cF^\infty_\varphi.
\]
It is called \emph{interpolating} for $\cF^\infty_\varphi$ if for every sequence
$v=(v_\lambda)_{\lambda\in\Lambda}$ satisfying
\[
\sup_{\lambda\in\Lambda}|v_\lambda|e^{-\varphi(\lambda)}<\infty,
\]
there exists $f\in\cF^\infty_\varphi$ such that
\[
f(\lambda)=v_\lambda, \quad \lambda\in\Lambda .
\]
{Finally,} $\Lambda$ is called a \emph{complete interpolating sequence} for $\cF^p_\varphi$, $0<p\leq \infty$,   if it is both sampling and interpolating.\\

When $p=2$, the Hilbert space structure allows one to interpret these notions in terms of the reproducing kernels  $K_\gamma$. Indeed, a complete interpolating sequence $\Gamma$ corresponds precisely to a Riesz basis of normalized reproducing kernels 
$\mathcal{K}_{\Gamma}:=\{K_\gamma/ \|K_\gamma\|_{\varphi,2}\}_{\gamma \in \Gamma} \subset \mathcal{F}^2_\varphi$.\\

In the setting of general sufficiently regular weights, one can describe conditions ensuring the existence of Riesz bases of reproducing kernels.  For a weight $\varphi$, set $\psi(t)=2\varphi(e^t).$
{Baranov, Belov and Borichev} \cite{BBB17} proved that  if $\psi$ satisfies suitable regularity conditions (see \eqref{conditions} below),
then $\mathcal{F}^2_\varphi$ admits a Riesz basis of normalized reproducing kernels supported on the real line. The basis points are given by
\[
\sigma_n = e^{y_n}, \quad \psi'(y_n)=2n+2.
\]
 They also showed that if $\psi''$ tends to infinity, then $\mathcal{F}^2_\varphi$ does not admit a Riesz basis of normalized reproducing kernels. See also the related works of Isaev and Yulmukhametov, where conditions on $\varphi$ are established to study the existence of Riesz bases of normalized reproducing kernels in weighted Fock spaces \cite{IY} and the references therein.
 
 \subsection{Reference sequence and complete interpolation {sequences} for $\cF^p_\varphi$}
Under these assumptions, the geometry of the space is determined by a radial reference sequence. Throughout the paper, we assume that
\[
\psi(t) = p \, \varphi(e^t)
\]
satisfies the above regularity conditions:
{\begin{equation}\label{conditions}
    \psi'(t) \to \infty, \quad \psi''(t) > 0 \text{ is non-increasing, and } 
|\psi'''(t)| = O\bigl((\psi''(t))^{5/3}\bigr), \quad t \to \infty.
\end{equation}

Before stating the main theorem, we introduce the reference sequence $\Sigma_{\varphi,p}$, which represents the ideal distribution of points in $\cF^p_\varphi$, for $0<p<\infty$. Namely, $\Sigma_{\varphi,p}$ is given by
\begin{equation}\label{sigma}
\Sigma_{\varphi,p}=\{\sigma_n=e^{y_n}: \ \psi'(y_n)=pn+2,\ n\ge0\}.
\end{equation}
This sequence} will serve as a model geometry for the space $\mathcal F^p_\varphi$.

{
We also define a logarithmic distance
\[
d_{\log} (z,w) = \frac{|z - w|}{1 + \min(|z|,|w|)}, \qquad z,w \in \mathbb{C}, 
\]
and   say that a sequence $\Gamma$ is \emph{logarithmically separated}, if there exists $d_\Gamma > 0$ such that
\[
\inf \{ d_{\log}(\gamma, \gamma^*) : \gamma, \gamma^* \in \Gamma,  \ \gamma \neq \gamma^* \} \ge d_\Gamma.
\]
}

For a sequence $\Gamma \subset \mathbb{C}$ {and $N$ a positive integer}, we define $\Delta_N(\Gamma)$ in \eqref{deltan} to measure the average radial deviation from the reference sequence over $N$ consecutive points. 
The following theorem characterizes complete interpolating sequences in terms of logarithmic separation and bounded radial perturbation relative to $\Sigma_{\varphi,p}$ given by an Avdonin-type condition.

\begin{theo}\label{thm1}
Let $\Gamma=\{\gamma_n\}_{n\ge0}$ be a sequence of $\C$ such that $|\gamma_n|\le |\gamma_{n+1}|$ and write $\gamma_n=e^{y_n}e^{\delta_n}e^{i\theta_n}$, for some real sequences $(\delta_n)$ and $(\theta_n)$. Then $\Gamma$ is a complete interpolating sequence for $\cF^p_\varphi$, $0<p<\infty$,  if and only if 
\begin{enumerate}
    \item $\Gamma$ is  logarithmically separated.
    \item $\left(\psi''(y_n)\delta_n\right)$ is bounded.
    \item  There exists a positive integer $N$ such that
     \begin{equation}\label{deltan}
        \Delta_N(\Gamma) := \limsup_{n\rightarrow\infty} \frac{1}{y_{n+N}-y_n} \left|\sum_{k=n+1}^{n+N}\delta_k\right|<\frac{1}{2}.
    \end{equation}
\end{enumerate}
\end{theo}

The following result characterizes complete interpolating sequences in $\mathcal F^\infty_\varphi$. This result reveals the deep link between complete interpolation in $\cF^\infty_\varphi$ and the Hilbertian structure of $\cF^2_\varphi$.

\begin{theo}\label{inftyThm}
Let $\Gamma \subset \C$ be a sequence and let $\gamma^* \in \C \setminus \Gamma$. Then $\Gamma \cup \{\gamma^*\}$ is a complete interpolating sequence for $\mathcal F^\infty_\varphi$ if and only if the family of reproducing kernels $\mathcal K_\Gamma=\{K_\gamma / \|K_\gamma\|_{\varphi,2}\}_{\gamma \in \Gamma}$ forms a Riesz basis {for} $\mathcal F^2_\varphi$.
\end{theo}

To illustrate these results, we now consider the special case 
$\varphi_2(z) = (\log^+ |z|)^2$, 
for which a complete description of complete interpolating sequences for 
$\cF^2_{\varphi_2}$ and $\cF^\infty_{\varphi_2}$ was obtained in \cite{BDHK}. 
For general $p \ge 1$, this was extended in \cite{O21}, while the case $0 < p < 1$ was treated by Mironov in \cite{Mi}. 
Belov and Mironov \cite{BM} also provided a complete characterization of shift-invariant sampling in terms of lower logarithmic density. 
The case {$\varphi_\alpha(z) = (\log^+ |z|)^\alpha$, where $1 < \alpha < 2$,} was studied in \cite{KO20} for $p=2$ and $p=\infty$.  
Baranov, Belov, and Gr\"ochenig \cite{BBG}, and more recently Baranov and Belov \cite{BB}, established a connection between Gabor systems and the space $\cF^2_{\varphi_2}$.

Finally, if a Fock-type space admits a Riesz basis of normalized reproducing kernels
$\mathcal{K}_\Gamma:=\{K_\gamma / \|K_\gamma\|_{\varphi,2}\}_{\gamma \in \Gamma}$ with $\Gamma\subset \R$, then it coincides (with equivalence of norms) with a de Branges space (see \cite{BMS}). All de Branges spaces that arise as Fock spaces were explicitly described in \cite{BBB2}.

\subsection{Density results}
{
Let $\cC(r, R)$ be the annulus centered at the origin with inner and outer radii $r$ and $R$, $
\cC(r, R) := \{ z \in \mathbb{C} : r \le |z| < R \}$.  For a given sequence $\Gamma \subset \mathbb C$, we denote by $n_\Gamma(B)$ the number of points of $\Gamma$ contained in the set $B$. Let $\eta(t) = (\psi')^{-1}(t)$. We then define the lower and upper $\varphi$-densities of $\Gamma$ as
\begin{equation}
\cD_\varphi^-(\Gamma) := \liminf_{R\to\infty} \inf_{r\ge 0} \frac{n_\Gamma\bigl(\cC(e^{\eta(r)},e^{\eta(r+R)})\bigr)}{R},
\end{equation}
and
\begin{equation}
\cD_\varphi^+(\Gamma) := \limsup_{R\to\infty} \sup_{r\ge 0} \frac{n_\Gamma\bigl(\cC(e^{\eta(r)},e^{\eta(r+R)})\bigr)}{R}. 
\end{equation}
}
\begin{theo}\label{densitythm}
Let  $0<p<\infty$ and $\Gamma$  be a logarithmically separated sequence.
\begin{enumerate}
\item\label{a} If $\cD_\varphi^+(\Gamma)<\frac{1}{p}$ then $\Gamma$ can be completed in such a way to become a complete interpolating sequence for $\cF^p_\varphi$.
\item\label{b} If $\cD_\varphi^-(\Gamma)>\frac{1}{p}$ then $\Gamma$ contains a complete interpolating sequence for $\cF^p_\varphi$.
\item\label{c} If $\Gamma$ is an interpolating sequence for $\cF^p_\varphi$ then  $\cD_\varphi^+(\Gamma)\leq \frac{1}{p}$.
\item\label{d} If $\Gamma$ is a sampling set for $\cF^p_\varphi$, then $\cD_\varphi^-(\Gamma)\geq \frac{1}{p}$.
\end{enumerate}
\end{theo}

\subsection*{Outline of the paper}The paper is organized as follows. 
Section \ref{sect2} contains some preliminary material, in particular estimates for the associated functional and bounds for canonical products. 
Section \ref{sect3} introduces the reference sequence and complete interpolating sequences for $\mathcal{F}^p_\varphi$, for $0<p\le \infty$. 
In Section \ref{sect4}, using the reference sequence and canonical product estimates, we establish the asymptotic evaluation estimates. 
Section \ref{sect5} deals with separation properties. Sections \ref{sect6}, \ref{sect7}, and \ref{sect8} contain the proofs of Theorems \ref{thm1}, \ref{inftyThm}, and \ref{densitythm}, respectively.\\

Throughout the paper, we denote by  $D(z,r)$ the disc centered at $z$ with radius $r$ and {$\cC(r, R)$ the} annulus centered at the origin with inner and outer radii $r$ and $R$, $
\cC(r, R) := \{ z \in \mathbb{C} : r \le |z| < R \}$. We set
\[
\psi(t) = \psi_p(t) := p\,\varphi(e^t),
\]
and we note that $\psi$ depends on the parameter $p$. Whenever necessary, we shall explicitly write $\psi_p$ in order to avoid any possible ambiguity. { The} sequence $(y_n)$ is defined by $\psi'(y_n) = pn + 2$, 
$\Sigma_{\varphi,p}$ is given by \eqref{sigma}, 
and we denote by $\eta(t) = (\psi')^{-1}(t)$ the inverse function of $\psi'$. 
The notation $A \lesssim B$ means that there exists an absolute constant $c>0$ such that $A \le c B$, 
while $A \asymp B$ means that both $A \lesssim B$ and $B \lesssim A$ hold.

\section{Preliminary Results}\label{sect2}

\subsection{First  evaluation estimates.}
We begin with the following lemmas, which provide estimates for the norm of the evaluation map
\[
L_z : f \in \mathcal{F}^p_\varphi \longmapsto f(z) \in \mathbb{C}.
\]
These results will be used repeatedly throughout the paper. 
An exact estimate of $\|L_z\|$ will be given later (see Theorem~\ref{kernelestimate}).

\begin{lem}\label{lem3.1}
  For every $p> 0$ we have     \[\|L_z\|^p\lesssim \frac{\left(\psi''(\log|z|)\right)^{1/2}}{1+|z|^2}e^{p\varphi(z)},\quad z\in\C. \]
\end{lem}
\begin{proof}
  Let $f\in\cF^p_\varphi$, and define 
  \[ F(r)= \int_0^{2\pi}|f(re^{i\theta})|^p\frac{d\theta}{\pi} = \exp(\omega(\log r)).\] 
Then we have
$$
        \|f\|^p_{\varphi,p}  = \int_\C |f(z)|^pe^{-p\varphi(z)}dm(z) 
        = \int_0^\infty F(r) e^{-p\varphi(r)} rdr
         = \int_\R e^{\omega(t)-\psi(t)+2t}dt.
$$
Let $u=\log|z|$.   Using  $\psi(t) = \psi(u) + \psi'(u)(t-u)+O(1)$ for $|t-u|\leq \varepsilon=:\left(\psi''(u)\right)^{-1/2}
  $,  we get
    \begin{equation}\label{11}
       \|f\|^p_{\varphi,p}  \geq   \int_{u-\varepsilon}^{u+\varepsilon} e^{\omega(t)-\psi(t)+2t}dt \gtrsim\  e^{-\psi(u)+u\psi'(u)} \int_{u-\varepsilon}^{u+\varepsilon} e^{\omega(t)-\psi'(u)t+2t}dt.
\end{equation}
Thanks to the convexity {of $\omega(t)$}, the function 
$$\displaystyle v(\varepsilon)=\frac{1}{\varepsilon}\int_{u-\varepsilon}^{u+\varepsilon} e^{\omega(t)-\psi'(u)t+2t} dt$$ is increasing in $(0,\infty)$ and hence for small positive $\kappa$ we get
\begin{equation*}
    \int_{u-\varepsilon}^{u+\varepsilon} e^{\omega(t)-\psi'(u)t+2t} dt \geq \frac{\varepsilon}{\kappa} \int_{u-\kappa}^{u+\kappa} e^{\omega(t)-\psi'(u)t+2t} dt.
\end{equation*}
Thus, 
\begin{align*}
      \|f\|^p_{\varphi,p}  & \gtrsim\ \left(\psi''(t)\right)^{-1/2} e^{-\psi(u)+u\psi'(u)} \int_{u-\kappa}^{u+\kappa} e^{\omega(t)-\psi'(u)t+2t}dt\\
      & = \left(\psi''(t)\right)^{-1/2} e^{-\psi(u)+u\psi'(u)}\int_{\cC(e^{u-\kappa}, e^{u+\kappa})} \frac{|f(\xi)|^p}{|\xi|^{\psi'(u)}} dm(\xi).
\end{align*}
Using, now, the mean inequality in  $D(z,\frac{\kappa|z|}{4})\subset \cC(e^{u-\kappa},e^{u+\kappa})$, we then get
\begin{align*}
    \frac{|f(z)|^p}{|z|^{\psi'(u)}} \lesssim \frac{1}{|z|^2}\int_{D(z,\frac{\kappa|z|}{4})} \frac{|f(\xi)|^p}{|\xi|^{\psi'(u)}} dm(\xi) \lesssim \frac{1}{|z|^2}\int_{\cC(e^{u-\kappa},e^{u+\kappa})} \frac{|f(\xi)|^p}{|\xi|^{\psi'(u)}} dm(\xi).
\end{align*}
Finally, we have
\[
|f(z)|^p \lesssim \frac{|z|^{\psi'(u)}}{|z|^2}\int_{\cC(e^{u-\varepsilon},e^{u+\varepsilon})} \frac{|f(\xi)|^p}{|\xi|^{\psi'(u)}} dm(\xi) \lesssim \frac{\left(\psi''(u)\right)^{1/2}}{|z|^2}e^{\psi(u)}\|f\|_\varphi^p.
\]
Hence
\[\|L_z\|^p\lesssim \frac{\left(\psi''(u)\right)^{1/2}}{1+|z|^2}e^{\psi(u)}  = \frac{\left(\psi''(\log|z|)\right)^{1/2}}{1+|z|^2}e^{p\varphi(z)}.\]
This completes the proof.

\end{proof}

We will need the following {lemma. The} case p = 2 was treated in \cite{BBB17}.
\begin{lem}\label{zn}
    We have for every positive integer $n$
    \begin{equation}
    \|z^n\|^p_{\varphi,p} \asymp \left(\psi''(y_n)\right)^{-1/2} e^{(pn+2)y_n-\psi(y_n)},
    \end{equation}
    and consequently, 
    \begin{equation}\label{Lsigma}
\left\|L_{e^{y_n}}\right\| \asymp \left(\psi''(y_n)\right)^{\frac{1}{2p}}e^{\frac{1}{p}\psi(y_n)-\frac{2}{p}y_n},\quad n\geq 0.
    \end{equation}
\end{lem}
\begin{proof}
    We  first have
    \begin{align*}
     \|z^n\|^p_{\varphi,p}  = \int_\C |z|^{np}e^{-p\varphi(z)}dm(z) \asymp \int_0^\infty e^{(pn+2)t-\psi(t)}dt
     \asymp \left(\psi''(y_n)\right)^{-1/2} e^{(pn+2)y_n-\psi(y_n)}.
    \end{align*}
    On the other hand, from the previous lemma we have $$\left\|L_{e^{y_n}}\right\| \lesssim \left(\psi''(y_n)\right)^{\frac{1}{2p}}e^{\frac{1}{p}\psi(y_n)-\frac{2}{p}y_n}.$$
  Now, using {the estimates} of the norm of the function $z\mapsto z^n$,  we obtain
    \begin{align*}
        \left\|L_{e^{y_n}}\right\|^p & \ge \frac{|e^{ny_n}|^p}{\|z^n\|^p_{\varphi,p}} \asymp \left(\psi''(y_n)\right)^{\frac{1}{2}}e^{pny_n}e^{-(np+2)y_n+\psi(y_n)} 
         \asymp \left(\psi''(y_n)\right)^{\frac{1}{2}}e^{\psi(y_n)-2y_n}.
    \end{align*}
\end{proof}

\begin{lem}
 Let $|z| = e^t$ and denote by $N_t = \lfloor (\psi')^{-1}(t) \rfloor$ the integer part of the inverse of $\psi'$ at $t$. Then
\[
\|L_z\|^p \gtrsim \bigl(\psi''(t)\bigr)^{1/2} \, 
e^{\, p N_t t + \psi(y_{N_t}) - (p N_t + 2) y_{N_t}}.
\]
\end{lem}
\begin{proof}  It suffices to consider the function
$
e_n(z)={z^n}/{\|z^n\|_{\varphi,p}},
$
since the computations are the same as those carried out at the end of Lemma \ref{zn}

\end{proof}

\subsection{Canonical products estimates} 
We begin with the following lemma
\begin{lem}\label{lem22} We have 
$$
   - p\sum_{k=1}^m y_k  =- \psi'(y_m)y_m +\psi(y_m)-\frac{p}{2}y_m(1+o(1)) + O(1).
$$
\end{lem}
\begin{proof}
 If $f$ is an increasing $C^1$- function, then 
$f(t) = f(m) + f'(m)(t-m)(1+o(1))
$, and 
\begin{align*}
    \int_{m-1}^m f(t)dt = f(m)+f'(m)\left(\frac{-1}{2}+o(1)\right),
\end{align*}
whence 
\begin{align*}
    \sum_{k=1}^m f(k) & = \int_0^mf(t) dt +\left(\frac{1}{2}+o(1)\right) \sum_{k=1}^m f'(m) \\
    & = \int_0^mf(t) dt +\left(\frac{1}{2}+o(1)\right) \int_0^mf'(t)dt   \\
    & = \int_0^mf(t) dt +\left(\frac{1}{2}+o(1)\right) (f(m)-f(0)).
\end{align*}
Applying this formula to the function $f(t)=(\psi')^{-1}(pt+2)$,  we get
\begin{align}\label{a21}
    \sum_{k=1}^m y_k & =  \sum_{k=1}^m (\psi')^{-1}(pk+2) 
     = \int_0^m(\psi')^{-1}(pt+2)dt + \frac{y_m}{2}(1+o(1))
     \end{align}
On the other hand, 
\begin{align}\label{a22}
    \int_0^{m} (\psi')^{-1}(pt+2)dt & = \int_{y_0}^{y_{m}} x\psi''(x)\frac{dx}{p} \nonumber\\
    & = \frac{y_{m}}{p}\psi'(y_{m}) -\frac{\psi(y_m)}{p} + O(1). %\\
\end{align}
Thus, we obtain the desired result  by \eqref{a21} and \eqref{a22}, which completes the proof of the lemma.

\end{proof}

Let 
\[
\Gamma :=\left\{\gamma_n:=e^{y_n+\delta_n}e^{i\theta_n}\quad :\ n\geq 0\right\}
\]
We {associate to $\Gamma$} the canonical product 
\begin{equation}\label{G-function}
    G_\Gamma(z) := \prod_{\gamma\in\Gamma} \left(1-\frac{z}{\gamma}\right),\quad z\in\C.
\end{equation}
For every positive integer $N$, we define the following quantity, provided that the limit exists 
\begin{equation}
    \Delta_N := \limsup_{n\rightarrow\infty} \frac{1}{y_{n+N}-y_n} \left|\sum_{k=n+1}^{n+N}\delta_k\right|.
\end{equation}
\begin{lem}\label{lem3.3}
The infinite product in \eqref{G-function} converges uniformly on every compact set of the complex plane $\C$ to an analytic function $G_\Gamma$. Moreover, for every positive integer $N$ and every $\varepsilon>0$, there exists $C=C(N,\varepsilon)\geq 1$ such that the following estimates hold:
\begin{equation}
    \label{estim-G-fct}
   \frac{1}{C}\frac{\dist(z,\Gamma)^p}{(1+|z|)^{2+\frac{p}{2}+p\Delta_N+\varepsilon}} \le  \left|G_\Gamma(z)\right|^p e^{-p\varphi(z)}\le C \frac{\dist(z,\Gamma)^p}{(1+|z|)^{2+\frac{p}{2}-p\Delta_N-\varepsilon}},\quad z\in\C, 
\end{equation}
and 
\begin{equation}\label{der}
    \frac{1}{C}\frac{1}{(1+|\gamma|)^{2+\frac{p}{2}+p\Delta_N+\varepsilon}} \le \left|G'_\Gamma(\gamma)\right|^pe^{-p\varphi(\gamma)} \le C \frac{1}{(1+|\gamma|)^{2+\frac{p}{2}-p\Delta_N-\varepsilon}},\quad \gamma\in\Gamma.
\end{equation}

\end{lem}

\begin{proof}
Let $z\in\C$ and denote by {$m$ the integer} such that $|\gamma_m|\leq |z|<|\gamma_{m+1}|$. Without loss of generality, we may assume that $\dist(z,\Gamma)=|\gamma_m-z|$. Let $t = \log |z|$, then 
\begin{align}
    \log|G_\Gamma(z)|^p & = p\sum_{n=0}^{m-1}\log\left|1-\frac{z}{\gamma_n}\right| + p\log\left|1-\frac{z}{\gamma_m}\right|+O(1)\nonumber\\
    & = p\sum_{n=0}^m\log\left|\frac{z}{\gamma_n}\right| + p\log\dist(z,\Gamma)-p\log|z| + O(1)\nonumber\\
    & = pmt - p\sum_{n=0}^my_n -p\sum_{n=0}^m\delta_n + p\log\dist(z,\Gamma) + O(1).\label{estim1}
\end{align}
On the other hand from Lemma \ref{lem22}, we have 
\begin{equation}\label{estim22}
   - p\sum_{k=1}^m y_k  =- \psi'(y_m)y_m +\psi(y_m)-\frac{p}{2}y_m(1+o(1)) + O(1).
\end{equation}
For $N$ a positive integer, there exists a positive integer $q$ and $0\le r<N$  such that $m=qN+r$. Then 
\begin{align}
    \left|\sum_{n=0}^m \delta_n \right| & = \left|\sum_{k=0}^{q-1}\sum_{n=kN+1}^{N(k+1)} \delta_n+\sum_{n=qN+1}^m \delta_n\right| \nonumber\\
    & \le \sum_{k=0}^{q-1}\left|\sum_{n=kN+1}^{N(k+1)} \delta_n\right|+\left|\sum_{n=qN}^m \frac{\delta_{n}\psi''(y_n)}{\psi''(y_n)}\right|\nonumber\\
    & \le \left(\Delta_N+o(1)\right)\sum_{k=0}^{q-1} (y_{kN+N}-y_{kN}) + o\left(y_m\right)\nonumber\\
    &  = \left(\Delta_N+o(1)\right)t.\label{estim3} 
\end{align}
Combining \eqref{estim1}, \eqref{estim22} and \eqref{estim3} we obtain
\begin{align*}
     \log|G_\Gamma(z)|^p & \le \psi'(y_m)(t-y_m)+\psi(y_m)+\left(-2-\frac{p}{2}+ p\Delta_N+o(1)\right)t + p\log\dist(z,\Gamma) + O(1) \\
     & = \psi(t) +\left(-2-\frac{p}{2}+ p\Delta_N+o(1)\right)t + p\log\dist(z,\Gamma) + O(1),
\end{align*}
and similarly,
\begin{align*}
     \log|G_\Gamma(z)|^p & \ge  \psi(t) +\left(-2-\frac{p}{2}- p\Delta_N+o(1)\right)t + p\log\dist(z,\Gamma) + O(1).
\end{align*}
Hence, we obtain
\begin{equation}
   \frac{\dist(z,\Gamma)^p}{(1+|z|)^{2+\frac{p}{2}+p\Delta_N+\varepsilon}}\lesssim  \left|G_\Gamma(z)\right|^p e^{-p\varphi(z)}\lesssim \frac{\dist(z,\Gamma)^p}{(1+|z|)^{2+\frac{p}{2}-p\Delta_N-\varepsilon}},\quad z\in\C.
\end{equation}
{The proof of \eqref{der} is similar and therefore we omit it here. }
\end{proof}

%%%%%%%%%%%%%%%%%%%%%%%%%%%%%%%%%%%%%%%%%%%%%%%%%%%%%%%%%%%%%%%%%%%
%%%%%%%%%%%%%%%%%%%%%%%%%%%%%%%%%%%%%%%%%%%%%%%%%%%%%%%%%%%%%%%%%%%
%%%%%%%%%%%%%%%%%%%%%%%%%%%%%%%%%%%%%%%%%%%%%%%%%%%%%%%%%%%%%%%%%%%

%%%%%%%%%%%%%%%%%%%%%%%%%%%%%%%%%%%%%%%%%%%%%%%%%%%%%%%%%
%%%%%%%%%%%%%%%%%%%%%%%%%%%%%%%%%%%%%%%%%%%%%%%%%%%%%%%%%
%%%%%%%%%%%%%%%%%%%%%%%%%%%%%%%%%%%%%%%%%%%%%%%%%%%%%%%%%

\begin{lem}\label{uniq}
    The sequence $\Gamma$ is a minimal uniqueness set {for $\cF^p_\varphi$, $0<p<\infty$,} whenever $\Delta_N<\frac{1}{2}$, for some positive integer $N.$
\end{lem}
\begin{proof}
  Suppose that $\Delta_N<\frac{1}{2}$, for some positive integer $N.$  If $F\in \cF^p_\varphi$ that vanishes on $\Gamma$. Then $F=hG_\Gamma$, for some entire function $h$.
    Using Lemma \ref{lem3.1} and Lemma \ref{lem3.3},  we obtain
    \[
    \frac{\left(\psi''(\log|z|)\right)^{1/2}}{1+|z|^2} \gtrsim |F(z)|^p e^{-p\varphi(z)} \gtrsim \frac{\dist(z,\Gamma)^p}{(1+|z|)^{2+\frac{p}{2}+p\Delta_N+\varepsilon}}|h(z)|^p.
    \]
    It follows that $h$ must be a polynomial. If it is not the zero function, we denote $m$ its degree's. Then
    \begin{align*}
       \infty >\int_\C |F(z)|^pe^{-p\varphi(z)}dm(z)\gtrsim \int_\C \frac{|z|^{p+pm}}{(1+|z|)^{2+\frac{p}{2}+p\Delta_N+\varepsilon}}dm(z),  
           \end{align*}
which implies $\displaystyle  \Delta_N+\frac{\varepsilon}{p}>\frac{1}{2}+m.$
This  is impossible, since  $\displaystyle\Delta_N < \frac{1}{2}$, hence $h$ is zero, and therefore $F=0$.

Now, if $\gamma\in\Gamma$, the sequence $\Gamma\setminus\{\gamma\}$ is a zero set for $\cF^p_\varphi$. Indeed, we have 
\begin{align*}
    \int_\C \left|\frac{G_\Gamma(z)}{z-\gamma}\right|^p e^{-p\varphi(z)} dm(z) & \lesssim \int_\C \frac{\dist(z,\Gamma)^p}{(1+|z|)^{p+2+\frac{p}{2}-p\Delta_N-\varepsilon}} dm(z) \\
    & \lesssim \int_\C \frac{1}{(1+|z|)^{2+\frac{p}{2}-p\Delta_N-\varepsilon}} dm(z). %\\
  \end{align*}
Since $\Delta_N<\frac{1}{2}$, we have  $2<2+\frac{p}{2}-p\Delta_N$ and hence  the  integral on the right-hand side converges. Therefore  $g_\gamma(z):= \frac{G_\Gamma(z)}{z-\gamma}$ belongs to  $\cF^p_\varphi$, which implies that $\Gamma\setminus\{\gamma\}$ is a zero set for $\cF^p_\varphi.$ This completes the proof.

 \end{proof}
\section{Reference complete interpolating sequence for $\cF^p_\varphi$.}\label{sect3}
To construct a complete reference sequence for $\cF^p_\varphi$, we first require the following lemma, which is taken from \cite[Lemma 3.3]{BBB17} for the case $p=2$. 
Let \begin{equation}\label{fonctionl}
   \ell(t) = (pm+2)t-p\sum_{n=0}^m y_k,\quad y_m\le t<y_{m+1}.
\end{equation}
\begin{lem}\label{lemBBB}
    Given a small positive $\delta$, the following hold:
    \begin{enumerate}
        \item $\displaystyle\int_0^\infty e^{\frac{1}{p}(\ell(t)-\psi(t))}dt<\infty.$
        \item $\displaystyle\int_{y_n-\delta}^\infty e^{\frac{1}{p}(\ell(t)-\psi(t))}dt \lesssim e^{\frac{1}{p}(\ell(y_n)-\psi(y_n))}\left(\psi''(y_n)\right)^{-1/2}$,
        \item $\displaystyle\int_0^{y_n-\delta}e^{\frac{1}{p}(\ell(t)-\psi(t)+pt)}dt\lesssim e^{\frac{1}{p}(\ell(y_n)-\psi(y_n)+py_n)}\left(\psi''(y_n)\right)^{-1/2}$
        \item $\displaystyle\sum_{k=0}^n e^{\frac{1}{p}(\psi(y_k)-\ell(y_k))}\left(\psi''(y_k)\right)^{1/2} \asymp e^{\frac{1}{p}(\psi(y_n)-\ell(y_n))}\left(\psi''(y_n)\right)^{1/2}$
        \item $\displaystyle\sum_{k=n}^\infty e^{\frac{1}{p}(\psi(y_n)-\ell(y_n)-py_n)}\left(\psi''(y_n)\right)^{1/2} \asymp e^{\frac{1}{p}(\psi(y_n)-\ell(y_n)-py_n)}\left(\psi''(y_n)\right)^{1/2}$.
    \end{enumerate}
\end{lem}

The following theorem extend the result obtained in \cite[Theorem 1.2]{BBB17} and give similar result as in \cite[Theorem 2.6]{O21}
\begin{theo}\label{theoRef}
Let $0<p\le \infty$. For $0<p<\infty$, set
\[
\Sigma_{\varphi,p} := \Big\{\sigma_n = e^{y_n} : \psi'(y_n) = pn + 2,\ n \ge 0\Big\},
\]
and for $p=\infty$ define
\[
\Sigma_{\varphi,\infty} := \Sigma_{\varphi,2} \cup \{\sigma^*\}, \qquad \sigma^* \notin \Sigma_{\varphi,2}.
\]
Then $\Sigma_{\varphi,p}$ is a complete interpolating sequence for $\cF^p_\varphi$.
\end{theo}
\begin{proof}
     The sequence $\Sigma_p=\Sigma_{\varphi,p}$ is a complete interpolating set for $\cF^p_\varphi$ if and only if $\Gamma$ is an interpolating and uniqueness set for $\cF^p_\varphi$. 
Let 
\begin{equation*}
    G_{\Sigma_p}(z) = \prod_{\sigma\in\Sigma_p} \left(1-\frac{z}{\sigma}\right),\quad z\in\C.
\end{equation*}
   Note that the sequence $\Sigma_p$ corresponds to $\Gamma$ with $\Delta_N=0$, and is therefore a uniqueness set for $\cF^p_\varphi${, $0<p<\infty$.}  is injective. 
   On the other hand, for every $v=(v_n)\in\ell^p$, we consider the function
    \begin{equation}\label{fv}
        f_v(z) := \sum_n v_n \|L_{\sigma_n}\| \frac{G_{\Sigma_p}(z)}{G'_{\Sigma_p}(\sigma_n)(z-\sigma_n)},\quad z\in\C.
    \end{equation}
    The above series converges uniformly on every compact set of the complex plane to an entire function $f_v$, which satisfies the interpolation problem $v_n={f_v(\sigma_n)}/{\|L_{\sigma_n}\|}$, for every integer $n$. We need to prove that $f_v$ belongs to $\cF^p_\varphi$. {Indeed, let 
     $\ell$ given by \eqref{fonctionl}. For every $z\in\C$, we write $|z|=e^t$. If $y_m\le t< y_{m+1}$, and suppose that $\dist(z,\Sigma_p)=|z-\sigma_m|$. By \eqref{estim1},  we have 
\begin{align*}
    \log|G_{\Sigma_p}(z)|^p 
    & = (pm+2)t-p\sum_{n=0}^m y_m + p\log\dist(z,\Sigma_p)-2t + O(1)\\
    & = \ell(t) + p\log\dist(z,\Sigma_p)-2t + O(1).
\end{align*}
It follows that 
\begin{equation}\label{ggg}
    |G_{\Sigma_p}(z)|^p \asymp e^{\ell(t)-2t} \dist(z,\Sigma_p)^p.
\end{equation}
Similarly, 
\begin{equation}\label{ggg1}
    |G'_{\Sigma_p}(\sigma_n)|^p \asymp e^{\ell(y_n)-2y_n},\quad n\geq 0.
\end{equation}
Using {the estimates of $\|L_{\sigma_n}\|$ in Lemma \ref{zn}, and \eqref{ggg} as well as \eqref{ggg1}}, we then obtain
\begin{align}\label{estgsigmA}
\left|\|L_{\sigma_n}\|\frac{G_{\Sigma_p}(z)}{G'_{\Sigma_p}(\sigma_n)(z-\sigma_n)}  \right|^p 
   & \asymp \left(\psi''(y_n)\right)^{\frac{1}{2}}e^{\psi(y_n)-\ell(y_n)}e^{\ell(t)-2t}\frac{\dist(z,\Sigma_p)^p}{|z-\sigma_n|^p}.
\end{align}
}
Consequently,   the function $f_v$, defined by \eqref{fv}, satisfies 
\begin{equation}\label{ser1}
    |f_v(z)|e^{-\varphi(z)} \lesssim \sum_n |v_n|\left(\psi''(y_n)\right)^{\frac{1}{2p}}e^{\frac{1}{p}\psi(y_n)-\frac{1}{p}\ell(y_n)}e^{\frac{1}{p}\ell(t)-\frac{1}{p}\psi(t)-\frac{2}{p}t}\frac{\dist(z,\Sigma_p)}{|z-\sigma_n|},\quad z\in\C.
\end{equation}

We  now distinguish the three cases: $0 < p \le 1$, $1 < p < \infty$, and $p = \infty$: \\

%%%%%%%%%%%%%%%%%%%%

\noindent {\bf $ \bullet$ The case  where $1<p<\infty$.}
First note that, by (4) and (5) of Lemma \ref{lemBBB}, we have
{\small
$$
\sum_{n=0}^m \left(\psi''(y_n)\right)^{\frac{1}{2p}}e^{\frac{1}{p}\psi(y_n)-\frac{1}{p}\ell(y_n)}  +|z|\sum_{n=m}^\infty\left(\psi''(y_n)\right)^{\frac{1}{2p}}e^{\frac{1}{p}\left(\psi(y_n)-\ell(y_n)-py_n\right)}       \asymp \left(\psi''(y_m)\right)^{\frac{1}{2p}}e^{\frac{1}{p}\psi(y_m)-\frac{1}{p}\ell(y_m)}.
$$}
Since $\psi'$ is a non-decreasing function and $\ell$ is  affine,  the function  $\psi-\ell$ is  convex. Hence  
\[
\psi(t)-\ell(t)\geq \psi(y_m)-\ell(m) + \left[\psi'(y_m)-\ell'(y_m)\right](t-y_m) = \psi(y_m)-\ell(m),\quad t\geq 0,
\]
which gives 
\begin{equation}
I(z) = \sum_n \left(\psi''(y_n)\right)^{\frac{1}{2p}}e^{\frac{1}{p}\psi(y_n)-\frac{1}{p}\ell(y_n)}\frac{\dist(z,\Sigma_p)}{|z-\sigma_n|} 
 \lesssim \left(\psi''(t)\right)^{\frac{1}{2p}}e^{\frac{1}{p}\psi(t)-\frac{1}{p}\ell(t)}.
\end{equation}
{ Using} Jensen inequality, we obtain
\begin{align}
    |f_v(z)|^pe^{-p\varphi(z)} & \lesssim \left[\sum_n |v_n|\left(\psi''(y_n)\right)^{\frac{1}{2p}}e^{\frac{1}{p}\psi(y_n)-\frac{1}{p}\ell(y_n)}e^{\frac{1}{p}\ell(t)-\frac{1}{p}\psi(t)-\frac{2}{p}t}\frac{\dist(z,\Sigma_p)}{|z-\sigma_n|}\right]^p\nonumber\\
    & \lesssim \sum_n |v_n|^p\left(\psi''(y_n)\right)^{\frac{1}{2p}}e^{\frac{1}{p}\psi(y_n)-\frac{1}{p}\ell(y_n)}e^{\ell(t)-\psi(t)-2t}\frac{\dist(z,\Sigma_p)}{|z-\sigma_n|}I(z)^{p-1}.\label{ser2}
\end{align}
{Let $q$ be} the H\"older conjugate of $p$.  For every small positive $\delta$, following $(2)$ of Lemma \ref{lemBBB}, we have
\begin{align*}
    J_{n}^{(2)} & := \int_{|z|\ge e^{-\delta}\sigma_n}e^{\frac{1}{p}\ell(\log^+|z|)-\frac{1}{p}\psi(\log^+|z|)-2\log^+|z|}\frac{\dist(z,\Sigma_p)}{|z-\sigma_n|}\left(\psi''(\log^+|z|)\right)^{\frac{1}{2q}} dm(z) \\
    & \lesssim \left(\psi''(y_n)\right)^{\frac{1}{2q}}\int_{|z|\ge e^{-\delta}\sigma_n}e^{\frac{1}{p}\ell(\log^+|z|)-\frac{1}{p}\psi(\log^+|z|)-2\log^+|z|}dm(z) \\
    & \lesssim \left(\psi''(y_n)\right)^{\frac{1}{2q}}\int_{y_n-\delta}^\infty e^{\frac{1}{p}\ell(t)-\frac{1}{p}\psi(t)} dt\\
    & \lesssim \left(\psi''(y_n)\right)^{\frac{1}{2q}}\left(\psi''(y_n)\right)^{\frac{-1}{2}}e^{\frac{1}{p}\ell(y_n)-\frac{1}{p}\psi(y_n)}\\
    & \lesssim \left(\psi''(y_n)\right)^{\frac{-1}{2p}}e^{\frac{1}{p}\ell(y_n)-\frac{1}{p}\psi(y_n)}.
\end{align*}
Similarly, by $(3)$ of Lemma \ref{lemBBB} we have
\begin{align*}
    J_{n}^{(1)} & := \int_{|z|\le e^{-\delta}\sigma_n}e^{\frac{1}{p}\ell(\log^+|z|)-\frac{1}{p}\psi(\log^+|z|)-2\log^+|z|}\frac{\dist(z,\Sigma_p)}{|z-\sigma_n|}\left(\psi''(\log^+|z|)\right)^{\frac{1}{2q}}dm(z) \\
    & \lesssim \frac{1}{\sigma_n}\int_{|z|\le e^{-\delta}\sigma_n} |z| e^{\frac{1}{p}\ell(\log^+|z|)-\frac{1}{p}\psi(\log^+|z|)-2\log^+|z|} dm(z)\\
    & \lesssim e^{-y_n}\int_0^{y_n-\delta}e^{\frac{1}{p}\ell(t)-\frac{1}{p}\psi(t)+t}\left(\psi''(t)\right)^{\frac{1}{2q}}dt\\
    & \lesssim \left(\psi''(y_n)\right)^{\frac{-1}{2p}}e^{\frac{1}{p}\ell(y_n)-\frac{1}{p}\psi(y_n)}.
\end{align*}
Hence, 
\begin{align}
    J_n & = \int_\C e^{\ell(\log^+|z|)-\psi(\log^+|z|)-2\log^+|z|}\frac{\dist(z,\Sigma_p)}{|z-\sigma_n|}I(z)^{p-1}dm(z) \nonumber\\
    &  = J_{n}^{(1)} + J_{n}^{(2)}   \lesssim \left(\psi''(y_n)\right)^{\frac{-1}{2p}}e^{\frac{1}{p}\ell(y_n)-\frac{1}{p}\psi(y_n)}.\label{2}
\end{align}
Integrating both sides of \eqref{ser2} with respect to the Lebesgue measure and  using the estimate \eqref{2}, we obtain 
\begin{align*}
    \|f_v\|_{\varphi,p}^p & \lesssim \sum_n |v_n|^p\left(\psi''(y_n)\right)^{\frac{1}{2p}}e^{\frac{1}{p}\psi(y_n)-\frac{1}{p}\ell(y_n)}J_n \lesssim \sum_n |v_n|^p=\|v\|_p^p.
\end{align*}

\noindent {\bf $ \bullet$ The case where  $0<p\le 1$}. By \eqref{ser1},  we have 
\begin{align*}
    \int_\C |f_v(z)|^pe^{-p\varphi(z)}dm(z) & \lesssim \sum_n |v_n|^p\left(\psi''(y_n)\right)^{\frac{1}{2}}e^{\psi(y_n)-\ell(y_n)}\int_\C e^{\ell(t)-\psi(t)-2t}\frac{\dist(z,\Sigma_p)^p}{|z-\sigma_n|^p}dm(z). 
\end{align*}
For  sufficiently small $\delta$, we have  
\begin{align*}
I_n  & = \int_\C e^{\ell(t)-\psi(t)-2t}\frac{\dist(z,\Sigma_p)^p}{|z-\sigma_n|^p}dm(z)\\
& = \int_{|z|\le e^{-\delta}\sigma_n} + \int_{|z|\ge e^{-\delta}\sigma_n}e^{\ell(t)-\psi(t)-2t}\frac{\dist(e^t,\Sigma_p)^p}{|e^t-\sigma_n|^p}e^tde^t\\
  & \lesssim e^{-py_n}\int_0^{y_n-\delta}e^{\ell(t)-\psi(t) +pt} dt + \int_{y_n-\delta}^\infty e^{\ell(t)-\psi(t)}dt \\
  & \lesssim e^{-py_n}\left(\psi''(y_n)\right)^{-1/2}e^{\ell(y_n)-\psi(y_n) +py_n}+\left(\psi''(y_n)\right)^{-1/2}e^{\ell(y_n)-\psi(y_n)}\\
  & \lesssim \left(\psi''(y_n)\right)^{-1/2}e^{\ell(y_n)-\psi(y_n)}.
\end{align*}
Therefore,
\[
 \int_\C |f_v(z)|^pe^{-p\varphi(z)}dm(z) \lesssim\ \sum_n |v_n|^p = \|v\|^p_p.\\
\]

\noindent {\bf $ \bullet$ The case  where $p=\infty$.} First, $\Sigma^*=\Sigma_2\cup \{\sigma^*\}$ is a uniqueness set for $\cF^\infty_\varphi$. Indeed, if $F\in \cF^\infty_\varphi$  vanishes on $\Sigma^*$,  then, we can write $F(z)=(z-\sigma^*)G_{\Sigma_2}(z)h(z)$ for some entire function $h$. The estimates of $G_{\Sigma_2}$ in Lemma \ref{lem3.3} (here $\Delta_N=0$) ensure that 
\[
1\gtrsim |F(z)|e^{-\varphi(z)} \gtrsim\frac{\dist(z,\Gamma)}{(1+|z|)^{\frac{1}{2}+\varepsilon}}|h(z)|,\quad z\in\C.
\]
Hence,  $h$ must be identically zero, and hence $F=0$.

For simplicity of notation, we denote $\sigma^*=\sigma_{-1}$ and write $\Sigma^*:=(\sigma_n)_{n\ge -1}$. For every $v=(v_n)\in\ell^\infty$, consider {the function} 
\begin{equation}
    f_v(z) = \sum_{n\ge -1} v_ne^{\varphi(\sigma_n)}\frac{G(z)}{G'(\sigma_n)(z-\sigma_n)},\quad z\in\C,
\end{equation}
where $G(z)=G_{\Sigma^*}(z)=(1-\frac{z}{\sigma_{-1}})G_{\Sigma_2}(z)$. Using identities \eqref{ggg} and \eqref{ggg1}, we obtain
\begin{align*}
\left|e^{\varphi(\sigma_n)}\frac{G(z)}{G'(\sigma_n)(z-\sigma_n)}  \right|^2 & \asymp e^{2\varphi(\sigma_n)-\ell(y_n)+\ell(t)} \frac{\dist(z,\Sigma^*)^2}{|z-\sigma_n|^2}.
\end{align*}
Using the notation  $\psi(t)=2\varphi(e^t)$, let { $n_t$ be} the unique integer that satisfies $y_{n_t} \le t\le y_{n_t+1}$, we get for $|z|=e^t$
\begin{align*}
    |f_v(z)|e^{-\varphi(z)} & \lesssim \sum_n |v_n|e^{\frac{1}{2}(\psi(y_n)-\ell(y_n)+\ell(t)-\psi(t))} \frac{\dist(z,\Sigma^*)}{|z-\sigma_n|}
 \lesssim \|v\|_\infty \Big({A(n_t)}  + {B(n_t)} \Big)
\end{align*}
where 
\begin{eqnarray*}
 A(n_t) &=& \sum_{n\le n_t}e^{\frac{1}{2}(\psi(y_n)-\ell(y_n)+\ell(t)-\psi(t))}\\
&=& \sum_{n\le n_t} \exp\Big[\frac{1}{2}\big(\psi(y_n)-\psi'(y_n)y_n + 2\sum_{k=0}^n y_k 
   + \psi'(y_{n_t})y_{n_t} - 2\sum_{k=0}^{n_t} y_k - \psi(t)\big)\Big] \\
  &\asymp& \sum_{n\le n_t} \exp\Big[\frac{1}{2}(-(y_{n_t}-y_n)(1+o(1)) + \psi(y_{n_t})-\psi(t))\Big] \\
  &\lesssim& \sum_{n\le n_t} e^{-c|y_{n_t}-y_n|} \lesssim 1.
\end{eqnarray*}
for some positive constant $c>0$, and 
\[
 B(n_t) =e^t\sum_{n\ge n_t+1}e^{\frac{1}{2}(\psi(y_n)-\ell(y_n)+\ell(t)-\psi(t)-2y_n)} \lesssim \sum_{n\ge n_t+1} e^{-c|y_n-y_{n_t}|} \lesssim 1.
\]
Consequently, 
\[
\sup_{z\in\C} |f_v(z)|e^{-\varphi(z)} \lesssim \|v\|_\infty.
\]
This proves that $\Sigma^*$ is a complete interpolating sequence for $\cF^\infty_\varphi$ and {concludes} the proof of Theorem \ref{theoRef}.
\end{proof}

\begin{cor}
    The following maps $\|.\|_{\Sigma_{\varphi,p}}$ given by
    \begin{equation}
        \|f\|_{\Sigma_{\varphi,p}}^p := \sum_{n\geq 0}\ \left|\frac{f(\sigma_n)}{\|L_{\sigma_n}\|}\right|^p 
    \end{equation}
    defines a norm if $p\ge 1$ and a semi-norm if $0<p< 1$ in $\cF^p_\varphi$ equivalent to $\|.\|_{\varphi,p}$. 
\end{cor}

{\section{Asymptotic evaluation estimates}\label{sect4}
In this section, we obtain asymptotic evaluation estimates for $\cF^p_\varphi$,  the linear maps 
$L_z\ :\ f\in\cF^p_\varphi\longmapsto f(z)\in\C$.  In particular, for $p=2$, we recover the known diagonal asymptotics of the reproducing kernel for weights of the form $\varphi(z) = (\log^+ |z|)^{\alpha}$, $1<\alpha\leq 2$, obtained in \cite{BDHK,KO20}.

\begin{theo}\label{kernelestimate}
    Let $\varphi$ be a weight satisfying {conditions \eqref{conditions}} and let $0<p<\infty$. 
        Let $\rho(z)=(\Delta \varphi(z))^{-1/2}$. 
    \begin{equation*}
            \|L_z\|^p \asymp \frac{e^{p\varphi(z)}}{|z| \rho(z)} \left[\exp\Big({-p  \int_{{\sigma_{n_z+1}}}^{|z|}  {r}{\Delta(\varphi(r))}  \, \log\frac{|z|}{r} \, dr}\Big)+ \exp\Big({-p  \int_{{\sigma_{n_z}}}^{|z|}{r}{\Delta(\varphi(r))} \, \log\frac{|z|}{r} \, dr}\Big)\right]
\end{equation*}          
    where $n_z$ is the unique integer such that $\sigma_{n_z}\le |z| <\sigma_{n_z+1}.$
\end{theo}
\begin{proof}
According to Theorem \ref{theoRef}, the sequence $\Sigma_p=\Sigma_{\varphi,p}$ is a complete interpolating set for $\cF^p_\varphi$, $0<p<\infty$. Hence every function $f$ of $\cF^p_\varphi$ can be written as 
\[
f(z) = \sum_n a_n \frac{\|L_{\sigma_n}\|}{G_{\Sigma_p}'(\sigma_n)}\frac{G_{\Sigma_p}(z)}{z - \sigma_n},\quad z\in\C
\]
for a unique sequence $(a_n)\in\ell^p$, and moreover  
$$\|f\|_{\varphi,p}\asymp\|(a_n)\|_p.$$ 
Consequently, 
\begin{align}
    \| L_z \| & =   \sup_{\|f\|_{\varphi,p}\le 1} |f(z)| =  \sup_{\|f\|_{\varphi,p}\le 1} \left|\sum_n a_n \frac{\|L_{\sigma_n}\|}{G_{\Sigma_p}'(\sigma_n)}\frac{G_{\Sigma_p}(z)}{z-\sigma_n}\right|\nonumber\\
     & \asymp \sup_{\|(a_n)\|_p\le 1} \left|\sum_n a_n \frac{\|L_{\sigma_n}\|}{G_{\Sigma_p}'(\sigma_n)}\frac{G_{\Sigma_p}(z)}{z - \sigma_n}\right|.
    \label{Lz} 
\end{align}

\noindent {\bf$\bullet$  Case $1<p<\infty$.} Using {the duality between $\ell^p$ and $\ell^q$}, where $\frac{1}{p}+\frac{1}{q}=1$, by \eqref{Lz}  we obtain
\begin{eqnarray}
\| L_z \| 
&\asymp&  \left( \sum_n \frac{\|L_{\sigma_n}\|^q}{|G_{\Sigma_p}'(\sigma_n)|^q}\frac{|G_{\Sigma_p}(z)|^q}{|z - \sigma_n|^q}\right)^{1/q} =S(z)^{1/q}. \label{lzz}
\end{eqnarray}
By Lemma \ref{zn}, \eqref{ggg} and \eqref{ggg1},  we have
\[
\|L_{\sigma_n}\|^p \asymp (\psi''(y_n))^{1/2} e^{\psi(y_n)-2y_n},\quad n\ge 0
\]
and 
\[
|G_{\Sigma_p}(z)|^p \asymp \dist(z,\Sigma_p)^p e^{\ell(t)-2t},\quad |G_{\Sigma_p}'(\sigma_n)|^p \asymp e^{\ell(y_n)-2y_n}.
\]
It follows that  
\begin{align*}
  S(z) 
    & \asymp e^{\frac{q}{p}(\ell(t)-2t)} \left(\sum_{n\le n_t} + \sum_{n\ge n_t+1} (\psi''(y_n))^{\frac{q}{2p}}e^{\frac{q}{p}(\psi(y_n)-\ell(y_n))}\frac{\dist(z,\Sigma_p)^q}{|z - \sigma_n|^q}\right),
\end{align*}
where $n_t$ is the unique integer such that $y_{n_t}\le t < y_{n_t+1}$. 
By Lemma \ref{lemBBB}
\begin{align*}
    \sum_{n\ge n_t+1}(\psi''(y_n))^{\frac{q}{2p}}e^{\frac{q}{p}(\psi(y_n)-\ell(y_n))}\frac{\dist(z,\Sigma_p)^q}{|z - \sigma_n|^q} & \asymp e^{qt} \sum_{n\ge n_t+1} (\psi''(y_n))^{\frac{q}{2p}}e^{\frac{q}{p}(\psi(y_n)-\ell(y_n)-py_n)}\\
    & \asymp e^{qt}(\psi''(y_{n_t+1}))^{\frac{q}{2p}}e^{\frac{q}{p}(\psi(y_{n_t+1})-\ell(y_{n_t+1})-py_{n_t+1})}
\end{align*}
Similarly,
\begin{align*}
    \sum_{n\le n_t} (\psi''(y_n))^{\frac{q}{2p}}e^{\frac{q}{p}(\psi(y_n)-\ell(y_n))}\frac{\dist(z,\Sigma)^q}{|z - \sigma_n|^q}  & \asymp  \sum_{n\le n_t} (\psi''(y_n))^{\frac{q}{2p}}e^{\frac{q}{p}(\psi(y_n)-\ell(y_n))} \\
    & \asymp (\psi''(y_{n_t}))^{\frac{q}{2p}}e^{\frac{q}{p}(\psi(y_{n_t})-\ell(y_{n_t}))}.
\end{align*}
{Consequently,   by \eqref{Lz} we get}
\begin{align*}
    \|L_z\|^p 
& \asymp (\psi''(t))^{\frac{1}{2}} e^{-2t}\left(e^{\psi'(y_{n_t+1})(t-y_{n_t+1})+\psi(y_{n_t+1})}+e^{\psi'(y_{n_t})(t-y_{n_t})+\psi(y_{n_t})}\right). 
\end{align*}

{$\bullet$ \bf Case $0<p\le 1$.} {We use the duality between $\ell^p$ and $\ell^\infty$ in the following form}
\begin{equation}
    \langle u,v\rangle := \sum_{n=0}^\infty u_n\overline{v_n},\quad u\in\ell^p,\ v\in\ell^\infty.
\end{equation}
{Together with} \eqref{Lz}, we get 
\begin{equation}
    \|L_z\|\asymp \sup_{\|(a_n)\|_p\le 1} \left|\sum_n a_n \frac{\|L_{\sigma_n}\|}{G_{\Sigma_p}'(\sigma_n)}\frac{G_{\Sigma_p}(z)}{z- \sigma_n}\right| =  \sup_n \frac{\|L_{\sigma_n}\|}{|G_{\Sigma_p}'(\sigma_n)|}\frac{|G_{\Sigma_p}(z)|}{|z - \sigma_n|}\label{kern}
\end{equation}
The estimates of $\|L_{\sigma_n}\|$ and the function $G_{\Sigma_p}$ ensure that  
\begin{align*}
   I_n(z) & := \frac{\|L_{\sigma_n}\|}{|G_{\Sigma_p}'(\sigma_n)|}\frac{|G_{\Sigma_p}(z)|}{|z - \sigma_n|}\  \asymp\ e^{\frac{1}{p}(\ell(t)-2t)} (\psi''(y_n))^{\frac{1}{2p}}e^{\frac{1}{p}(\psi(y_n)-\ell(y_n))}\frac{\dist(z,\Sigma_p)}{|z - \sigma_n|}.
\end{align*}
If $m\ge n_{t}+1$, using that the functions $t\longmapsto e^{\psi(t)-\ell(t)-pt}$ and $\psi''$ are non-increasing, we obtain  
\begin{align*}
  (\psi''(y_m))^{\frac{1}{2p}}e^{\frac{1}{p}(\psi(y_m)-\ell(y_m))}\frac{\dist(z,\Sigma_p)}{|z - \sigma_m|} &  \lesssim  (\psi''(y_{n_t+1}))^{\frac{1}{2p}}|z|e^{\frac{1}{p}(\psi(y_m)-\ell(y_m)-py_m)}\\ 
  & \lesssim (\psi''(y_{n_t+1}))^{\frac{1}{2p}}e^{\frac{1}{p}(\psi(y_{n_t+1})-\ell(y_{n_t+1})-py_{n_t+1})+t}.
\end{align*}
Similarly, if $m\le n_t$, since $t\longmapsto e^{\psi(t)-\ell(t)}$ is non-decreasing, we obtain 
\begin{align*}
  (\psi''(y_m))^{\frac{1}{2p}}e^{\frac{1}{p}(\psi(y_m)-\ell(y_m))}\frac{\dist(z,\Sigma_p)}{|z - \sigma_m|} & \asymp (\psi''(y_m))^{\frac{1}{2p}}e^{\frac{1}{p}(\psi(y_m)-\ell(y_m))} \\
  & \lesssim (\psi''(y_{n_t}))^{\frac{1}{2p}}e^{\frac{1}{p}(\psi(y_{n_t})-\ell(y_{n_t}))}.
\end{align*}
Consequently,
\begin{align*}
       \|L_z\|^p & \lesssim        e^{-2t}\left((\psi''(y_{n_t+1}))^{\frac{1}{2}}e^{\psi'(y_{n_t+1})(t-y_{n_t+1})+\psi(y_{n_t+1})}+(\psi''(y_{n_t}))^{\frac{1}{2}}e^{\psi'(y_{n_t})(t-y_{n_t})+\psi(y_{n_t})}\right).
\end{align*}
Let $e_n(z)={z^n}/{\|z^n\|_{\varphi,p}}$. Clearly we have
$
|e_{n_t}(z)|^p
+
|e_{n_t+1}(z)|^p
\le
2^p\,\|L_z\|^p.
$
Therefore,
\[
\|L_z\|^p
\asymp
e^{-2t}
(\psi''(t))^{1/2}
\Big(
e^{\psi'(y_{n_t+1})(t-y_{n_t+1})+\psi(y_{n_t+1})}
+
e^{\psi'(y_{n_t})(t-y_{n_t})+\psi(y_{n_t})}
\Big).
\]
We obtain the same estimate in both cases; to conclude,
it remains to use   $\psi''(t)=p e^{2t}\Delta\varphi(e^t)$ and apply  Green's formula
 \[
\exp\!\left(
-p\int_{e^{y}}^{|z|}
r\,\Delta\varphi(r)\,\log\frac{|z|}{r}\,dr
\right)
=
\exp\!\bigl(
-\psi(t)+\psi(y)+\psi'(y)(t-y)
\bigr),
\quad t=\log|z|.
\]

The proof is now complete.
\end{proof}
}

\section{Separated sequences}\label{sect5}

A sequence $\Gamma$ is said to be  \emph{logarithmically separated} if there exists a constant $d_\Gamma > 0$ such that
\[
\inf \{ d_{\log}(\gamma, \gamma^*) : \gamma, \gamma^* \in \Gamma,  \ \gamma \neq \gamma^* \} \ge d_\Gamma, 
\]
where
\[
d_{\log} (z,w) = \frac{|z - w|}{1 + \min(|z|,|w|)}, \qquad z,w \in \mathbb{C}.
\]

One of the main tools we rely on is the following Bernstein-type result (see Lemma \ref{bernstein}).  
We begin with the following lemma, which shows that the inequality below allows us to prove that the sequence is a finite union of logarithmically separated sequences.
{
\begin{lem}\label{ufsp} Let $\Gamma=\{\gamma_n\ :\ n\ge 0\}$ be a sequence of $\C$ such that $\gamma_n=e^{y_n}e^{\delta_n}e^{i\theta_n}$, for some real sequences $(\delta_n)$  and $(\theta_n),$ and $|\gamma_n|\le |\gamma_{n+1}|$. There exists $C(\Gamma)>0$ such that 
\[
\sum_{\gamma \in \Gamma} \frac{|f(\gamma)|^p}{\|L_\gamma\|^{p}}
\leq C(\Gamma)\, \|f\|^p_{\varphi,p}, 
\quad f \in \mathcal{F}^p_{\varphi},
\]
if and only if $\Gamma$ is a finite union of logarithmic-separated sequences.
\end{lem}

\begin{proof}
Let 
$\Sigma_p = \Sigma_{\varphi,p} = \{\sigma_n = e^{y_n} e^{i\theta_n} : \psi'(y_n) = pn + 2,\ n\ge 0 \}$,
the reference sequence for $\cF^p_\varphi$. Consider
\[
g_{\sigma}(z)
=
\|L_{\sigma}\|\,
\frac{G_{\Sigma_p}(z)}{G'_{\Sigma_p}(\sigma)\,(z-\sigma)},
\quad z \in \mathbb{C}.
\]
We have $\|g_\sigma\|_{\varphi,p} \asymp 1$. For every $z\in\C$, we write $|z|=e^t$. If $y_m\le t< y_{m+1}$, and suppose that $\dist(z,\Sigma_p)=|z-\sigma_m|$. We have by \eqref{estgsigmA} 
\begin{align}\label{estgsigma}
\left|\|L_{\sigma_n}\|\frac{G_{\Sigma_p}(z)}{G'_{\Sigma_p}(\sigma_n)(z-\sigma_n)}  \right|^p 
   & \asymp \left(\psi''(y_n)\right)^{\frac{1}{2}}e^{\psi(y_n)-\ell(y_n)}e^{\ell(t)-2t}\frac{\dist(z,\Sigma_p)^p}{|z-\sigma_n|^p}.
\end{align}
Hence, for every small $\varepsilon>0$
\[
\left|g_{\sigma_n}(z)  \right|^p \asymp \left(\psi''(y_n)\right)^{\frac{1}{2}}e^{\psi(y_n)-\ell(y_n)+\ell(t)-2t},\quad z\in D(\sigma_n,\varepsilon|\sigma_n|).
\]
On the other hand, note that $D(\sigma_n,\varepsilon|\sigma_n|)\subset \cC(e^{y_n-2\varepsilon},e^{y_n+2\varepsilon})$ and hence for $z\in D(\sigma_n,\varepsilon|\sigma_n|)$ we have  
\[
\ell(t)-\ell(y_n) = \psi'(y_n)(t-y_n),\quad y_n \le t< y_{n+1},
\]
if $y_n\le t\le y_n+2\varepsilon$, and 
\begin{align*}
    \ell(t)-\ell(y_n) & = \psi'(y_{n-1})t-p\sum_{k=0}^{n-1}y_k-\psi'(y_n)y_n+p\sum_{k=0}^ny_k  \\ 
    & = \psi'(y_{n-1})t+py_n-(pn+2)y_n\\
    & = \psi'(y_{n-1})(t-y_n),\quad y_{n-1} \le t< y_{n}
\end{align*}
whenever $y_n-2\varepsilon\le t\le y_n$. Furthermore, by Theorem \ref{kernelestimate}  we have 
$$ \|L_z\|^p\asymp |e_n(z)|^p\asymp |z^n|^p/\|z^n\|^p_{\varphi,p} \asymp \left(\psi''(y_n)\right)^{\frac{1}{2}}e^{pnt+\psi(y_n)-\psi'(y_n)y_n}.$$
Hence
\begin{equation}
    \frac{\left|g_{\sigma_n}(z)  \right|^p}{\|L_z\|^p} \asymp e^{-\ell(y_n)+\ell(t)-2t-pnt+\psi'(y_n)y_n},\quad z\in D(\sigma_n,\varepsilon|\sigma_n|).
\end{equation}
It follows that 
\[
\ell(t)-\ell(y_n)-2t-pnt+\psi'(y_n)y_n = \begin{cases}
    0,\quad  y_n\leq t\leq y_n+2\varepsilon, \\
    (\psi'(y_n)-\psi'(y_{n-1}))(y_n-t),\quad y_n-2\varepsilon\leq t\leq y_n.
\end{cases}
\]
Thus \[
\ell(t)-\ell(y_n)-2t-pnt+\psi'(y_n)y_n\ge 0,\quad z\in D(\sigma_n,\varepsilon|\sigma_n|),
\]
which  ensures that $\frac{\left|g_{\sigma_n}(z)  \right|^p}{\|L_z\|^p} \gtrsim 1$, for every $z\in D(\sigma_n,\varepsilon|\sigma_n|)$. Hence
\begin{align*}
1 & \gtrsim 
\sum_{\gamma\in\Gamma} 
\frac{|g_\sigma(\gamma)|^{p}}{\|L_\gamma\|^{p}} \ \gtrsim 
\sum_{\gamma\in\Gamma \cap D(\sigma, \varepsilon\sigma)}
\frac{|g_\sigma(\gamma)|^{p}}{\|L_\gamma\|^{p}} \
\gtrsim 
\operatorname{Card}(\Gamma \cap D(\sigma, \varepsilon\sigma)).
\end{align*}
Finally
\[
\sup_{\sigma\in\Sigma_p} \operatorname{Card}(\Gamma \cap D(\sigma, c\sigma)) < \infty.
\]
Hence, $\Gamma$ is a finite union of logarithmic separated sequences.\\

Conversely, since $\Sigma_p$ is a complete interpolating sequence for $\cF^p_\varphi$ then every $f\in\cF^p_\varphi$ can be written as 
\[
f(z) = \sum_{n\ge 0} a_n \|L_{\sigma_n}\| \frac{G_{\Sigma_p}(z)}{G'_{\Sigma_p}(\sigma_n)(z-\sigma_n)},\quad z\in\C,
\]
for some unique $(v_n)\in\ell^p$ and $\|f\|_{\varphi,p}\asymp\|(a_n)\|_p$. Thus
\begin{align*}
    \sum_m\frac{|f(\gamma_m)|^p}{\|L_{\gamma_m}\|^p} & = \sum_m \left|\sum_{n\ge 0} a_n \frac{\|L_{\sigma_n}\|}{\|L_{\gamma_m}\|} \frac{G_{\Sigma_p}(\gamma_m)}{G'_{\Sigma_p}(\sigma_n)(\gamma_m-\sigma_n)} \right|^p \lesssim \|(a_n)\|^p_p
\end{align*}
The last inequality holds if and only if the matrix 
\[
a_{m,n} = \frac{\|L_{\sigma_n}\|}{\|L_{\gamma_m}\|} \frac{G_{\Sigma_p}(\gamma_m)}{G'_{\Sigma_p}(\sigma_n)(\gamma_m-\sigma_n)}
\]
maps continuously $\ell^p$ on itself. Let $|\gamma_m|=e^{x_m}$ and denote by $s_m$ the unique integer such that $y_{s_m}\le x_m< y_{s_m+1}$. Using \eqref{ggg} for $z=\gamma_m$, \eqref{ggg1}, Lemma \ref{zn} and Theorem \ref{kernelestimate},  we obtain
\begin{equation}
    |a_{m,n}|^p \asymp \frac{(\psi''(y_n))^{1/2}}{(\psi''(y_{s_m}))^{1/2}}e^{\vartheta(m,n)}\frac{|\gamma_m-\sigma_{s_m}|^p}{|\gamma_m-\sigma_n|^p}
    \end{equation}
where \[
\vartheta(m,n) = \ell(x_m)-2x_m-ps_mx_m-\psi(y_{s_m})+\psi'(y_{s_m})y_{s_m}+\psi(y_n)-\ell(y_n).
\]
Using the definition of $\ell$ as well as Lemma \eqref{lem22}, we obtain
\begin{align*}
   \vartheta(m,n) & = -p\sum_{k=0}^{s_m}y_k -\psi(y_{s_m})+\psi'(y_{s_m})y_{s_m}+\psi(y_n)-\psi'(y_n)y_n+p\sum_{k=0}^ny_k\\
   & = -\frac{p}{2}(y_{s_m}-y_n)(1+o(1)).
\end{align*}
If $s_m\ge n$,  we have $|\gamma_m-\sigma_{s_m}|=\dist(\gamma_m,\Sigma_p)\le |\gamma_m-\sigma_n|$ 
\[
|a_{m,n}|\lesssim e^{-c|y_m-y_n|}.
\]
If $s_m<n$, we have $|\gamma_m-\sigma_n|\asymp|\sigma_n|$ and
$|\gamma_m-\sigma_{s_m}|\le|\gamma_m|$, which gives again
\[
|a_{m,n}|^p \lesssim \frac{(\psi''(y_n))^{1/2}}{(\psi''(y_{s_m}))^{1/2}} e^{-\frac{p}{2}(y_{s_m}-y_n)(1+o(1))}e^{py_{s_m}-py_n}\lesssim e^{-c|y_m-y_n|}.
\]
Therefore  if $\Gamma$ is a finite union of logarithmically separated sequences then 
\[
\sum_m\left|\sum_n a_n a_{m,n}\right|^p \lesssim \sum_m\left(\sum_n |a_n| e^{-c|y_{s_m}-y_n|}\right)^p \lesssim \|(v_n)\|^p_p \asymp \|f\|^p_{\varphi,p}.
\]
This completes the proof.
\end{proof}
}
{

\begin{lem}\label{bernstein}
Let $0<p<\infty$ and let $\varepsilon>0$ be a sufficiently small parameter. 
Let $z\in \C \setminus \Sigma_p$. Then for every 
$w\in D(z,\varepsilon |z|)\setminus \Sigma_p$ we have
\[
\left|\frac{f(z)}{G_{\Sigma_p}(z)}-\frac{f(w)}{G_{\Sigma_p}(w)}\right|
\lesssim 
d_{\log}(z,w)\,
\frac{\|L_z\|}{|G_{\Sigma_p}(z)|}\,
\|f\|_{\varphi,p},
\]
for every $f\in \mathcal{F}^p_\varphi$.
\end{lem}
\begin{proof}
    Since $\Sigma_p=\Sigma_{\varphi,p}$ is a complete interpolating sequence for $\mathcal F^p_\varphi$. every function $f\in\mathcal{F}^p_\varphi$ can be written as 
\[
f(z)=\sum_n a_n \|L_{\sigma_n}\| \frac{G_{\Sigma_p}(z)}{G'_{\Sigma_p}(\sigma_n)(z-\sigma_n)}
\]
for a unique sequence $(a_n)\in\ell^p$ and such that $\|f\|_{\varphi,p}\asymp\|(a_n)\|_p$. It follows that
\begin{align}
    \left|\frac{f(z)}{G_{\Sigma_p}(z)}-\frac{f(w)}{G_{\Sigma_p}(w)}\right| & =\left|\sum_n a_n \frac{\|L_{\sigma_n}\|}{G'_{\Sigma_p}(\sigma_n)}\left(\frac{1}{(z-\sigma_n)}-\frac{1}{w-\sigma_n}\right)\right|\nonumber\\
    & = \left|\sum_n a_n \frac{\|L_{\sigma_n}\|}{G'_{\Sigma_p}(\sigma_n)}\frac{z-w}{(z-\sigma_n)(w-\sigma_n)}\right|\nonumber\\
    & \lesssim |z-w| \sum_n |a_n| \frac{(\psi''(y_n))^{\frac{1}{2p}}}{|z-\sigma_n||w-\sigma_n|}e^{\frac{1}{p}(\psi(y_n)-\ell(y_n))}. \label{berns}.
\end{align}
{\bf$\bullet$ The case where $1<p<\infty$.} Let $q$ the H\"older conjugate of $p$. Using H\"older's inequality we obtain
\begin{align*}
    \left|\frac{f(z)}{G_{\Sigma_p}(z)}-\frac{f(w)}{G_{\Sigma_p}(w)}\right| 
        & \lesssim |z-w| \|f\|_{\varphi,p} \left[\sum_n \frac{(\psi''(y_n))^{\frac{q}{2p}}}{|z-\sigma_n|^q|w-\sigma_n|^q}e^{\frac{q}{p}(\psi(y_n)-\ell(y_n))}\right]^{1/q}.
\end{align*}
Let $m$ the unique integer such that $\sigma_m\le |z|<\sigma_{m+1}$, and assume that $w\in D(z,\varepsilon|z|)$. Then $D(z,\varepsilon|z|)\subset \mathcal{C}(\sigma_{m-1},\sigma_{m+1})$. Write 
\begin{align*}
    I(z,w) & = \sum_n \frac{(\psi''(y_n))^{\frac{q}{2p}}}{|z-\sigma_n|^q|w-\sigma_n|^q}e^{\frac{q}{p}(\psi(y_n)-\ell(y_n))} \\
    & = \sum_{n\le m-1} + \sum_{n\ge m+2}\frac{(\psi''(y_n))^{\frac{q}{2p}}}{|z-\sigma_n|^q|w-\sigma_n|^q}e^{\frac{q}{p}(\psi(y_n)-\ell(y_n))} + \frac{(\psi''(y_m))^{\frac{q}{2p}}}{|z-\sigma_m|^q|w-\sigma_m|^q}e^{\frac{q}{p}(\psi(y_m)-\ell(y_m))} \\ & \qquad + \frac{(\psi''(y_{m+1}))^{\frac{q}{2p}}}{|z-\sigma_{m+1}|^q|w-\sigma_{m+1}|^q}e^{\frac{q}{p}(\psi(y_{m+1})-\ell(y_{m+1}))} \\
    & = I_1 + I_2 + \frac{(\psi''(y_m))^{\frac{q}{2p}}}{|z-\sigma_m|^q|w-\sigma_m|^q}e^{\frac{q}{p}(\psi(y_m)-\ell(y_m))} + \frac{(\psi''(y_{m+1}))^{\frac{q}{2p}}}{|z-\sigma_{m+1}|^q|w-\sigma_{m+1}|^q}e^{\frac{q}{p}(\psi(y_{m+1})-\ell(y_{m+1}))}, 
\end{align*}
where
\begin{align*}
I_1&:= \sum_{n\le m-1}
\frac{(\psi''(y_n))^{\frac{q}{2p}}}
{|z-\sigma_n|^q|w-\sigma_n|^q}
e^{\frac{q}{p}(\psi(y_n)-\ell(y_n))}\\
&\asymp \frac{1}{|z|^{2q}}
\sum_{n\le m-1}
(\psi''(y_n))^{\frac{q}{2p}}
e^{\frac{q}{p}(\psi(y_n)-\ell(y_n))}\\
&\lesssim \frac{1}{|z|^{2q}}
\sum_{n\le m}
(\psi''(y_n))^{\frac{q}{2p}}
e^{\frac{q}{p}(\psi(y_n)-\ell(y_n))}\\
&\asymp 
\frac{(\psi''(y_m))^{\frac{q}{2p}}}{|z|^{2q}}
e^{\frac{q}{p}(\psi(y_m)-\ell(y_m))}.
\end{align*}

Similarly
\begin{align*}
I_2
&:=\sum_{n\ge m+2}
\frac{(\psi''(y_n))^{\frac{q}{2p}}}
{|z-\sigma_n|^q|w-\sigma_n|^q}
e^{\frac{q}{p}(\psi(y_n)-\ell(y_n))}\\
&\asymp
\sum_{n\ge m+2}
(\psi''(y_n))^{\frac{q}{2p}}
e^{\frac{q}{p}(\psi(y_n)-\ell(y_n))-2qy_n}.
\end{align*}
Consequently 
\begin{align*}
   \left|\frac{f(z)}{G_{\Sigma_p}(z)}\right.&\left.-\frac{f(w)}{G_{\Sigma_p}(w)}\right| \lesssim |z-w|\|f\|_{\varphi,p}\left(\frac{(\psi''(y_m))^{\frac{1}{2p}}}{|z|^{2}}e^{\frac{1}{p}(\psi(y_m)-\ell(y_m))}\right.\\ &  + (\psi''(y_{m+1}))^{\frac{1}{2p}}e^{\frac{1}{p}(\psi(y_{m+1})-\ell(y_{m+1})-2py_{m+1})} 
    +\frac{(\psi''(y_m))^{\frac{1}{2p}}}{|z-\sigma_m||w-\sigma_m|}e^{\frac{1}{p}(\psi(y_m)-\ell(y_m))}\\ &  \left. + \frac{(\psi''(y_{m+1}))^{\frac{1}{2p}}}{|z-\sigma_{m+1}||w-\sigma_{m+1}|}e^{\frac{1}{p}(\psi(y_{m+1})-\ell(y_{m+1}))}\right)\\
    & \lesssim d_{\log}(z,w)\|f\|_{\varphi,p}\left(\frac{(\psi''(y_m))^{\frac{1}{2p}}}{|z|}e^{\frac{1}{p}(\psi(y_m)-\ell(y_m))} + e^{-y_{m+1}}(\psi''(y_{m+1}))^{\frac{1}{2p}}e^{\frac{1}{p}(\psi(y_{m+1})-\ell(y_{m+1}))}\right)\\
    & \lesssim d_{\log}(z,w)\frac{\|L_z\|}{|G_{\Sigma_p}(z)|}\|f\|_{\varphi,p}.
\end{align*}
{\bf $\bullet$ The case where $0<p\le1$.} Relation \eqref{berns} gives 
\begin{align*}
    \left|\frac{f(z)}{G_{\Sigma_p}(z)}-\frac{f(w)}{G_{\Sigma_p}(w)}\right|^p
    & \lesssim |z-w|^p \sum_n |a_n|^p \frac{(\psi''(y_n))^{\frac{1}{2}}}{|z-\sigma_n|^p|w-\sigma_n|^p}e^{\psi(y_n)-\ell(y_n)}\\
    & \lesssim |z-w|^p\|f\|^p_{\varphi,p} \sup_n \left(\frac{(\psi''(y_n))^{\frac{1}{2}}}{|z-\sigma_n|^p|w-\sigma_n|^p}e^{\psi(y_n)-\ell(y_n)}\right)
\end{align*}
Following \eqref{kern}, we obtain 
\[
\left|\frac{f(z)}{G_{\Sigma_p}(z)}-\frac{f(w)}{G_{\Sigma_p}(w)}\right| \lesssim d_{\log}(z,w) \frac{\|L_z\|}{|G_{\Sigma_p}(z)|}\|f\|_{\varphi,p}.
\]
 The proof is therefore complete.
\end{proof}

\begin{cor}\label{cor22}
Every set of interpolation for $\mathcal{F}^p_\varphi$, {$0<p\le \infty$},  is logarithmically separated sequence.
\end{cor}

\begin{proof}
   If $\Gamma$ is an interpolating sequence, then for every $k$ there exists $f_k\in\mathcal F^p_\varphi$ such that 
\[
f_k(\gamma_k) = \|L_{\gamma_k}\|, \quad f_k(\gamma_j)=0, \quad \|f_k\|_{\varphi,p}\lesssim 1.
\]
It follows that 
\[
\frac{\|L_{\gamma_k}\|}{|G_{\Sigma_p}(\gamma_k)|}\lesssim d_{\log}(\gamma_k,\gamma_j)\frac{\|L_{\gamma_k}\|}{|G_{\Sigma_p}(\gamma_k)|}.
\]
Hence \(d_{\log}(\gamma_k,\gamma_j)\gtrsim 1\) and hence $\Gamma$ is logarithmically separated.
\end{proof}

\section{Proof of Theorem \ref{thm1}}\label{sect6}

\subsection{Sufficient conditions}
    The sequence $\Gamma$ is a complete interpolating set for $\cF^p_\varphi$ if and only if the operator
    $$\begin{array}{cccc}
        T_{\Gamma}: & \cF^p_\varphi & \longrightarrow & \ell^p \\
          & f & \longmapsto & \left(f(\gamma_n)/\|L_{\gamma_n}\|\right) 
    \end{array}$$
    is bounded and invertible from $\cF^p_\varphi$ to $\ell^p$.

{Since $\Gamma$ is logarithmically separated then by Lemma \ref{ufsp} $T_\Gamma$ is bounded from $\cF^p_\varphi$ to $\ell^p$.}  By Lemma \ref{uniq}, $\Gamma$ is a uniqueness set for $\cF^p_\varphi$ under the condition $\Delta_N<\frac{1}{2}$, hence,  $T_{\Sigma_p}$ is injective. On the other hand, fo every $v=(v_n)\in\ell^p$,   consider the function
    \begin{equation}
        f_v(z) := \sum_n v_n \|L_{\gamma_n}\| \frac{G_{\Gamma}(z)}{G'_\Gamma(\gamma_n)(z-\gamma_n)},\quad z\in\C.
    \end{equation}
    The estimates of the function $G_\Gamma$ ensure that the last series converges uniformly on every compact set of the complex plane to an entire function $f_v$, which satisfies the interpolation condition ${f_v(\gamma_n)}/{\|L_{\gamma_n}\|}=v_n$, for every integer $n$. It remains to prove that $f_v\in \cF^p_\varphi$. Indeed, 
    \begin{align*}
        \|f_v\|_{\Sigma_{\varphi,p}}^p & = \sum_{m\ge 0} \left|\sum_n v_n \frac{ \|L_{\gamma_n}\|}{\|L_{\sigma_m}\|} \frac{G_{\Gamma}(\sigma_m)}{G'_\Gamma(\gamma_n)(\sigma_m-\gamma_n)} \right|^p.
    \end{align*}
    The last quantity is finite, for every $v\in\ell^p$, if and only if the matrix 
    \begin{equation}
        A_{n,m} :=  \frac{ \|L_{\gamma_n}\|}{\|L_{\sigma_m}\|} \frac{G_{\Gamma}(\sigma_m)}{G'_\Gamma(\gamma_n)(\sigma_m-\gamma_n)}
    \end{equation}
    defines a bounded operator on  $\ell^p$.

     Recall the estimates of the function $G_\Gamma$
    \begin{equation*}
        \left|G_\Gamma(\sigma_m)\right|\asymp \frac{|\sigma_m-\gamma_{n_m}|}{|\gamma_{n_m}|}\prod_{k=0}^{n_m-1}\left|\frac{\sigma_m}{\gamma_k}\right|,\quad \left|G'_\Gamma(\gamma_n)\right| \asymp \frac{1}{|\gamma_n|}\prod_{k=0}^{n-1}\left|\frac{\gamma_n}{\gamma_k}\right|,
    \end{equation*}
where the integer $n_m$ is the unique integer for which $|\gamma_{n_m}|\le \sigma_m <|\gamma_{n_m+1}|$.

{ Since $(\psi''(y_n)\delta_n)$ is bounded, there exists $M>0$ such that $
|\psi''(y_n)\delta_n|\le M $ for all  $n$.
From
\[
y_{n_m}+\delta_{n_m} \le y_m < y_{n_m+1}+\delta_{n_m+1}, 
\]
the monotonicity of $\psi'$ and $\psi'(y_m)=pm+2$, we obtain
\[
\psi'(y_{n_m}+\delta_{n_m}) \le pm+2 \le \psi'(y_{n_m+1}+\delta_{n_m+1}).
\]
Hence 
\[
 pn_m+\psi''(y_{n_m})\delta_{n_m}(1+o(1))\le pm\le p(n_{m}+1)+\psi''(y_{n_m+1})\delta_{n_m+1}(1+o(1)) \\
\]
and therefore 
\[
 |m-n_m|\le 2M.
 \]
}

 In particular,  $n_{m}=m+j$ for some integer $j\in\left[-\frac{M}{p},\frac{M}{p}+1\right]$. 
Let $x_n=y_n+\delta_n$,  we obtain
\begin{align*}
   |A_{n,m}|^p & \lesssim \frac{\left(\psi''(x_n)\right)^{\frac{1}{2}}}{\left(\psi''(y_m)\right)^{\frac{1}{2}}}\frac{|\sigma_m-\gamma_{m+j}|^p}{|\sigma_m-\gamma_n|^p}\frac{\sigma_m^{p(m+j)}}{|\gamma_n|^{pn}}\left(\prod_{k=0}^{m+j}\frac{1}{|\gamma_k|^p}\right) \left(\prod_{k=0}^n|\gamma_k|^p\right)e^{\psi(x_n)-\psi(y_m)-2x_n+2y_m}\\
    & \lesssim  
    \frac{\left(\psi''(x_n)\right)^{\frac{1}{2}}}{\left(\psi''(y_m)\right)^{\frac{1}{2}}}\frac{|\sigma_m-\gamma_{m+j}|^p}{|\sigma_m-\gamma_n|^p}\times e^{\theta(n,m)},
\end{align*}
where
\begin{align*}
     \theta(n,m)
     & = \psi(y_n+\delta_n)-\psi(y_m)+ \psi'(y_{m+j})y_m-\psi'(y_n)(y_n+\delta_n)-p\sum_0^{m+j}y_k \\ &\quad +p\sum_{0}^{n}y_k-p\sum_0^{m+j}\delta_k+p\sum_{0}^{n}\delta_k.
\end{align*}
By Lemma \ref{lem22}, we know that
\[
-p\sum_0^m y_k = \psi(y_m)-\psi'(y_m)y_m-y_m\left(\frac{p}{2}+o(1)\right) + O(1).
\]
It follows that
\begin{align*}
     \theta(n,m)     & = \psi(y_n+\delta_n)-\psi(y_n)-\psi'(y_n)\delta_n+\psi(y_{m+j})-\psi(y_m)+\psi'(y_{m+j})(y_m-y_{m+j})\\
     & \qquad+(y_n-y_{m+j})\left(\frac{p}{2}+o(1)\right) +p\left(\sum_0^{n}\delta_k-\sum_0^{m+j}\delta_k\right)\\
     & = \frac{1}{2}\left(\psi''(y_n)\delta_n^2-\psi''(y_{m+j})(y_m-y_{m+j})^2\right)+(y_n-y_{m+j})\left(\frac{p}{2}+o(1)\right) \\
     & \qquad+p\left(\sum_0^{n}\delta_k-\sum_0^{m+j}\delta_k\right).
\end{align*}
And then
\begin{align*}
    \theta(n,m) & = (y_n-y_{m+j})\left(\frac{p}{2}+o(1)\right)  +(p+o(1))\left(\sum_0^{n}\delta_k-\sum_0^{m+j}\delta_k\right).
\end{align*}
If $n\le m+j$, then  $\dist(\sigma_m,\Gamma)=|\sigma_m-\gamma_{m+j}|\leq |\sigma_m-\gamma_n|$. \\
Recall that
\[
\Delta_N = \limsup_{n \to \infty} \frac{1}{y_{n+N} - y_n} \left| \sum_{k=n+1}^{n+N} \delta_k \right|.
\]
Thus, for every $\varepsilon>0$, there exists $n_0$ such that for all $n \ge n_0$,
\[
\left| \sum_{k=n+1}^{n+N} \delta_k \right| \le (\Delta_N + \varepsilon) \, (y_{n+N} - y_n).
\]
It follows that
\begin{align*}
    \sum_0^{n}\delta_k-\sum_0^{m+j}\delta_k = \sum_{n+1}^{m+j}\delta_k \le (\Delta_N+\varepsilon)(y_{m+j}-y_n).
\end{align*}
Consequently
$$
    \left|A_{n,m}\right|^p \lesssim \exp\left(-\left[\frac{p}{2}-p\Delta_N-p\varepsilon\right]|y_{m+j}-y_n|\right)$$
    if and only if 
    $$ \left|A_{n,m}\right| \lesssim \exp\left(-\left[\frac{1}{2}-\Delta_N-\varepsilon\right]|y_{m+j}-y_n|\right).
$$
A similar estimate holds in the case $m+j< n$, 
$$\frac{|\sigma_m-\gamma_{m+j}|^p}{|\sigma_m-\gamma_n|^p}\lesssim e^{(p+o(1))(y_{m+j}-y_n)}.$$  Therefore,
\begin{equation}
    \left|A_{n,m}\right| \lesssim \exp\left(-\left[\frac{1}{2}-\Delta_N-\varepsilon\right]|y_{m+j}-y_n|\right).
\end{equation}
Finally, if $\Delta_N<\frac{1}{2}$ then the matrix $(A_{n,m})$ maps continuously $\ell^p$ on itself.\\

\subsection{Necessary conditions}
Let $0< p<\infty$ and suppose that $\Gamma$ is a complete interpolating sequence for $\cF^p _\varphi$.

\medskip

\noindent
{\bf 1.} Since $\Gamma$ is a complete interpolating sequence for $\cF^p_\varphi$ then it is an interpolating sequence for $\cF^p_\varphi$ and hence by Corollary \ref{cor22}, $\Gamma$ is logarithmically separated.
\medskip

\noindent{\bf 2.}  We will prove that $(\psi''(y_n)\delta_n)$ is bounded.

Since $\Gamma$ is a complete interpolating set for $\cF^p_\varphi$, then for every $\gamma\in\Gamma$ there exists a unique function $f_\gamma\in \cF^p_\varphi$ vanishing on $\Gamma\setminus\{\gamma\}$ and satisfying 
\[ f_\gamma(\gamma)=1,\quad \mbox{and}\quad \|f_\gamma\|_{\varphi,p}\asymp\frac{1}{\|L_\gamma\|}.\]
Using Hadamard factorization theorem,  $f_\gamma(z)=\alpha \frac{G_\Gamma(z)}{z-\gamma}$, for some constant $\alpha\in\C$. Since $f_\gamma(\gamma)=1$, we obtain $\alpha=1$, hence $f_\gamma(z)=\frac{G_\Gamma(z)}{z-\gamma}.$  By continuity of the point evaluation map $L_z$ on $\cF^p_\varphi$ and the estimates of $\|L_z\|$ (see Lemma \ref{lem3.1}) we obtain
\begin{equation}\label{id1}
    |f_\gamma(z)|^p \lesssim \|f_\gamma\|^p_{\varphi,p}\|L_z\|^p\lesssim    \frac{\left(\psi''(\log|z|)\right)^{\frac{1}{2}}}{\|L_\gamma\|^p(1+|z|^2)}e^{p\varphi(z)},\quad z\in\C.
\end{equation}
Assume now that the sequence $\left(\psi''(y_n)\delta_n\right)$ contains a subsequence $\left(\psi''(y_{n_k})\delta_{n_k}\right)$  tending to $+\infty$ (the case $-\infty$ is similar). For every $k$ there exists $m_k\in\N$ such that $\gamma_{n_k}$ is close to $\sigma_{m_k}$ and $|n_k-m_k|\rightarrow\infty$. We may assume   that $|\gamma_{n_k}|\leq \sigma_{m_k}<|\gamma_{n_k+1}|$.\\ 
Recall the  estimates:
\begin{equation}\label{id2}
    \left|G_\Gamma(\sigma_{m_k})\right|\asymp\frac{|\sigma_{m_k}-\gamma_{n_k}|}{|\gamma_{n_k}|}\prod_{j=0}^{n_k-1}\frac{\sigma_{m_k}}{|\gamma_j|},\qquad \left|G'_\Gamma(\gamma_{n_k})\right|\asymp\frac{1}{|\gamma_{n_k}|}\prod_{j=0}^{n_k-1}\frac{|\gamma_{n_k}|}{|\gamma_j|}.
\end{equation}
Combining \eqref{id1} and \eqref{id2}, let $x_n=y_n+\delta_n$, for $z=\sigma_{m_k}$ yields to
\begin{align}
e^{pn_k(y_{m_k}-x_{n_k})} = 
  \left(\frac{\sigma_{m_k}}{|\gamma_{n_k}|}\right)^{pn_k}\asymp \left|f_{\gamma_{n_k}}(\sigma_{m_k})\right|^p \lesssim \frac{\left(\psi''(y_{m_k})\right)^{\frac{1}{2}}}{\|L_{\gamma_{n_k}}\|^p}e^{\psi(y_{m_k})-2y_{m_k}}\label{id3}.
\end{align}
On the other hand,
\[
\|L_{\gamma_{n_k}}\|^p \ge \frac{|\gamma_{n_k}|^{pm_k}}{\|z^{m_k}\|^p_{\varphi,p}}\asymp \left(\psi''(y_{m_k})\right)^{\frac{1}{2}}e^{pm_kx_{n_k}+\psi(y_{m_k})-(pm_k+2)y_{m_k}}.
\]
Hence by  \eqref{id3}, we obtain 
\begin{align*}
     e^{pn_k(y_{m_k}-x_{n_k})} & \lesssim 
     e^{pm_k(y_{m_k}-x_{n_k})} \Longleftrightarrow e^{p(n_k-m_k)(y_{m_k}-x_{n_k})}\lesssim 1.
\end{align*}
Since $n_k-m_k\rightarrow\infty$, this is impossible. Hence  $(\psi''(y_n)\delta_n)$ is bounded.\\

\medskip

\noindent
{\bf 3.} We will prove that $\Delta_N<\frac{1}{2}$ for some integer $N$. Assume, by contradiction, that for every $N\geq 1$,  $\Delta_N=\frac{1}{2}+\kappa_N$, for some  $\kappa_N\ge 0$.\\

Since $\Gamma$ is a complete interpolating sequence, then the matrix 
\begin{equation}
  A_{n,m} :=  \frac{ \|L_{\gamma_n}\|}{\|L_{\sigma_m}\|} \frac{G_{\Gamma}(\sigma_m)}{G'_\Gamma(\gamma_n)(\sigma_m-\gamma_n)}
 \end{equation}
 defines a bounded (invertible) operator on  $\ell^p$. 
 
Since $\left(\psi''(y_n)\delta_n\right)$ is bounded, there exists $M>0$ such that for every non negative integer $m$ we have $\dist(\sigma_m,\Gamma)=|\sigma_m-\gamma_{m+j}|$, for some integer $|j|\le M$. Using the standard estimates for $G_\Gamma$, $G'_\Gamma$, $\|L_{\sigma_m}\|$, and $\|L_{\gamma_n}\|$, we have 
\[
\left|G_\Gamma(\sigma_m)\right|\asymp\frac{|\sigma_m-\gamma_{m+j}|}{|\gamma_{m+j}|}\prod_{k=0}^{m+j-1}\left|\frac{\sigma_m}{\gamma_k}\right|,\quad \left|G'_\Gamma(\gamma_n)\right|\asymp\frac{1}{|\gamma_{n}|}\prod_{k=0}^{n-1}\left|\frac{\gamma_n}{\gamma_k}\right|, 
\]
\[
\|L_{\sigma_m}\|^p\asymp \left(\psi''(y_m)\right)^{1/2}e^{\psi(y_m)-2y_m}, \quad 
\|L_{\gamma_n}\|^p\geq \frac{e^{pnx_n}}{\|z^{n}\|^p_{\varphi,p}}\asymp \left(\psi''(y_{n})\right)^{1/2}e^{\psi(y_{n})-\psi'(y_{n})y_{n}+pnx_n}. 
\]
Consequently,
\begin{align*}
    |A_{n,m}|^p & \asymp \frac{\|L_{\gamma_n}\|^p}{\|L_{\sigma_m}\|^p}\frac{\dist(\sigma_m,\Gamma)^p}{|\sigma_m-\gamma_n|^p}e^{p(m+j)y_m-pn(y_n+\delta_n)}e^{p\sum_{k=0}^n(y_k+\delta_k)-p\sum_{k=0}^{m+j}(y_k+\delta_k)}\\
    & \gtrsim \frac{\left(\psi''(x_n)\right)^{1/2}}{\left(\psi''(y_m)\right)^{1/2}} \frac{\dist(\sigma_m,\Gamma)^p}{|\sigma_m-\gamma_n|^p}e^{\Theta(n,m)},
\end{align*}
where 
\begin{align*}
    \Theta(n,m) &:=\psi(y_{n})-\psi'(y_{n})y_{n}+pnx_n-\psi(y_m)+2y_m+p(m+j)y_m-pnx_n\\&\quad+p\sum_{k=0}^n(y_k+\delta_k)-p\sum_{k=0}^{m+j}(y_k+\delta_k).
\end{align*}
On the other hand, as in Lemma \ref{lem22}
\begin{equation}
   - p\sum_{k=1}^m y_k  =- \psi'(y_m)y_m +\psi(y_m)-\frac{p}{2}y_m - \left(\frac{p^2}{12}+o(1)\right)\frac{1}{\psi''(y_m)}+ O(1).
\end{equation}
This estimates imply that 
\begin{align*}
    \Theta(n,m) & 
        =
       \psi(y_{m+j})-\psi(y_m)-\psi'(y_{m+j})(y_{m+j}-y_m)+\frac{p}{2}(y_n-y_{m+j}) \\&\quad +\left(\frac{p^2}{12}+o(1)\right)\left(\frac{1}{\psi''(y_n)}-\frac{1}{\psi''(y_{m+j})}\right)\pm p\sum_{k=m+j+1}^n\delta_k\\
           & \ge \frac{p}{2}(y_n-y_{m+j})+C_\psi\left(\frac{1}{\psi''(y_n)}-\frac{1}{\psi''(y_{m+j})}\right)\pm p\sum_{k=m+j+1}^n\delta_k.
\end{align*}
It follows that
\begin{equation}\label{eAnm}
|A_{n,m}|^p \gtrsim \exp\left[-\frac{p}{2}|y_n-y_{m+j}|+C_\psi\left(\frac{1}{\psi''(y_n)}-\frac{1}{\psi''(y_{m+j})}\right)\pm p\sum_{k=m+j+1}^n\delta_k\right].
\end{equation}
Suppose now that 
$$
\Delta_N{=\limsup_{n}\frac{1}{y_{n+N}-y_n}\left|\sum_{k=n+1}^{n+N}\delta_k\right|}=\frac{1}{2}+\kappa_N,
$$
for some non negative sequence $(\kappa_N)$. It follows that for every integer $N$ there exists $n_N$ sufficiently large such that 
\[
\left|\sum_{k=n_N+1}^{n_N+N}\delta_k\right| \ge \frac{1+2\kappa_N}{2}(y_{n_N+N}-y_{n_N})-1.
\]
Two cases occur:

\medskip
\noindent
$\bullet$ If there exists a sequence of integers $(N_l)$ such that 
$$\sum_{k=n_l+1}^{n_l+N_l}\delta_k>0$$ (where, for simplicity, $n_l$ denotes the integer $n_{N_l}$). Then by  \eqref{eAnm}, we obtain  
    \begin{align*}
        |A_{n_l+N_l,n_l}|^p 
        & \gtrsim \exp\left[C_\psi\left(\frac{1}{\psi''(y_{n_l+N_l})}-\frac{1}{\psi''(y_{n_l})}\right)\right]
    \end{align*}
This shows that the sequence  $(A_{n_l+N_l,N_l})$ is unbounded and hence $(A_{n,m})$ cannot define a bounded operator on $\ell^p$.\\

\medskip
\noindent
$\bullet$  Suppose now that there exists $N_0$ such that  
$$\sum_{k=n_N+1}^{n_N+N}\delta_k<0,$$ for every $N\ge N_0$. Then  by   \eqref{eAnm}, we obtain 
      \begin{align*}
        |A_{n_N,n_N+N}|^p & \gtrsim \exp\left[-C_\psi\left(\frac{1}{\psi''(y_{n_N+N})}-\frac{1}{\psi''(y_{n_N})}\right)+\frac{p\kappa_{N}}{2}(y_{n_N+N}-y_{n_N})\right]
    \end{align*}
    If $(\kappa_{N})$ contains a subsequence bounded below by some $\varepsilon>0$, we denote this subsequence again by $(\kappa_N)$ for simplicity. Then 
    \begin{align*}
        |A_{n_N,n_N+N}|^p & \gtrsim \exp\left[-C_\psi\left(\frac{1}{\psi''(y_{n_N+N})}-\frac{1}{\psi''(y_{n_N})}\right)+\frac{p\varepsilon}{2}(y_{n_N+N}-y_{n_N})\right]
    \end{align*}
    This implies that $(A_{n_N,n_N+N})_N$ is unbounded and hence $A$ is not bounded in $\ell^p$.
    
Finally, assume that $(\kappa_N)\to 0$. Since for every convex function  $f$ (with $f'$ non-decreasing)  $\frac{t}{2}f'(t/2)\le f(t)$, we obtain  
\[\frac{n_N}{2}\left(\frac{1}{\psi''(n_N/2+N)}-\frac{1}{\psi''(y_{n_N}/2)}\right)\le y_{n_N+N}-y_{n_N}.\]
It follows that 
\begin{align*}
  |A_{n_N,n_N+N}  |^p 
    \gtrsim \exp\left[-C_\psi\left(\frac{1}{\psi''(y_{n_N+N})}-\frac{1}{\psi''(y_{n_N})}\right)+\frac{p\kappa_Nn_N}{4}\left(\frac{1}{\psi''(y_{\frac{n_N}{2}+N})}-\frac{1}{\psi''(y_{\frac{n_N}{2}})}\right)\right].
\end{align*}
By choosing $n_N$ sufficiently large, the right-hand side becomes unbounded, and hence $(A_{n_N,n_N+N})$. Consequently the matrix $A$ cannot define a bounded operator on $\ell^p$.

%%%%%%%%%%%%%%%%%%%%%%%%%%%%%%%%%%%%%%%
%%%%%%%%%%%%%%%%%%%%%%%%%%%%%%%%%%%%%%%%%%
%%%%%%%
\section{Proof of Theorem \ref{inftyThm}.}\label{sect7}
The statements proved in this section are essentially the same as in \cite{BDHK,KO20}. We include the proof here for completeness.\\

Suppose that $\cK_\Gamma=\{K_\gamma/\|K_\gamma\|_{\varphi,2}\}_{\gamma\in \Gamma}$ is a Riesz basis for $\cF^2_\varphi$ and let $\gamma^*\in\C\setminus \Gamma$. We will prove that $\Gamma^*:=\Gamma\cup\{\gamma^*\}$ is a complete interpolating sequence for $\cF^\infty_\varphi$.

First, if $F$ is a function of $\cF^\infty_\varphi$ that vanishes on $\Gamma^*$ then $F(z)=(z-\gamma^*)G_\Gamma(z)h(z)$ for some entire function $h$. According to Lemma \ref{lem3.3}, we have 
\[
1\gtrsim |F(z)|e^{-\varphi(z)} \gtrsim\frac{\dist(z,\Gamma)}{(1+|z|)^{\frac{1}{2}+\Delta_N+\varepsilon}}|h(z)|,\quad z\in\C.
\]
Together with the condition $\Delta_N<\frac{1}{2}$, ensure that $h$ must be identically zero. Hence $\Gamma^*$ is a uniqueness set for $\cF^\infty_\varphi.$

It remains to prove that $\Gamma^*$ is an interpolating sequence for $\cF^\infty_\varphi$. For simplicity of notations, we denote  $\gamma^*=\gamma_{-1}$ and let $v=(v_n)_n\in \ell^\infty$. Consider the function 
\begin{equation}
    F_v(z) = \sum_{n\ge -1} v_n e^{\varphi(\gamma_n)}\frac{G_{\Gamma^*}(z)}{G_{\Gamma^*}'(\gamma_n)(z-\gamma_n)}
\end{equation}
The  series converges uniformly on every compact set of $\C$ to an entire function that solves the interpolation problem : $F_v(\gamma_n)e^{-\varphi(\gamma_n)}=v_n$, for every $n\ge -1$. {We only need to show that} $F_v\in\cF^\infty_\varphi$. For this we observe that
\begin{align*}
    \|F_v\|_{\varphi,\infty} & \asymp \sup_m \left|e^{-\varphi(\sigma_m)} \sum_n v_n e^{\varphi(\gamma_n)}\frac{G_{\Gamma^*}(\sigma_m)}{G_{\Gamma^*}'(\gamma_n)(\sigma_m-\gamma_n)} \right| \asymp \sup_m\left|\sum_n v_n B_{n,m}\right|,
\end{align*}
where 
\begin{align*}
    |B_{n,m}| 
     & \asymp e^{\varphi(\gamma_n)-\varphi(\sigma_m)}\frac{|\sigma_m|}{|\gamma_n|}\left|\frac{G_{\Gamma}(\sigma_m)}{G_{\Gamma}'(\gamma_n)(\sigma_m-\gamma_n)}\right| = |A_{n,m}|e^{o(1)|y_n-y_m|}.
\end{align*}
The estimates for $(A_{n,m})$ imply that $(B_{n,m})$ defines a bounded operator on $\ell^\infty$ and then the desired result follows. \\

\medskip

Conversely, suppose $\Gamma^*=\Gamma\cup\{\gamma^*\}$ is a complete interpolating sequence for $\cF^\infty_\varphi$. To prove that $\cK_\Gamma$ is a Riesz basis for $\cF^2_\varphi$,  it suffices to show that $(1)-(3)$ of Theorem \ref{thm1} hold.

{First, by Corollary \ref{cor22}, since $\Gamma^*$ is a complete interpolating set for $\cF^\infty_\varphi$ then it is logarithmically separated. By consequence, condition $(1)$ is satisfied.\\

On the other hand, recall} that 
\begin{align*}
    |B_{n,m}| & \asymp e^{\varphi(\gamma_n)-\varphi(\sigma_m)}\frac{|\sigma_m|}{|\gamma_n|}\left|\frac{G_{\Gamma}(\sigma_m)}{G_{\Gamma}'(\gamma_n)(\sigma_m-\gamma_n)}\right| \ge \frac{(\psi''(y_m))^{1/2}}{(\psi''(x_n))^{1/2}} |A_{n,m}|.
\end{align*}
Following the lines of the proof of Theorem \ref{thm1}, the unboundedness of the sequence $(\psi''(y_n)\delta_n)$ implies that $(A_{n,m})$,  and therefore  $(B_{n,m})$ are unbounded. Hence, $(B_{n,m})$ cannot define a bounded operator
on $\ell^\infty$, and  condition $(2)$ follows.\\

Moreover, if 
$$\Delta_N = \frac{1}{2} + \kappa_N$$ for some non-negative sequence $(\kappa_N)$, then $(A_{n,m})$ is unbounded, and consequently so is $(B_{n,m})$. Thus, $(B_{n,m})$ cannot represent a bounded operator on $\ell^\infty$, which implies that condition $(3)$ is verified.

\section{Proof of Theorem \ref{densitythm}}\label{sect8}
To prove Theorem \ref{densitythm}, we adopt an approach inspired by \cite{BDHK,S95}.  
To establish the necessary conditions for a sequence to be sampling or interpolating, we rely on the following lemma. The proof in our setting is identical to that of \cite[Lemma 5.1]{BDHK} (see also \cite[Lemma 40]{NMO}), using  the point evaluation norm estimates and the canonical product bounds derived above.

\begin{lem}\label{Lemma 5.1}
Let $\varepsilon>0$. Assume that the sequence $\Lambda$ is interpolating for $\mathcal{F}^p_{(1-\varepsilon)\varphi}$, $0<p<\infty$, and that $\Sigma$ is sampling for $\mathcal{F}^2_\varphi$ and logarithmically separated. 
Then, for sufficiently small $\delta>0$ and sufficiently large $R$, we have
\[
(1-c(\delta)) \, \mathrm{Card}\bigl(\Lambda \cap \mathcal{C}(x,Rx)\bigr) \le \mathrm{Card}\bigl(\Sigma \cap \mathcal{C}(\delta x, Rx/\delta)\bigr),
\]
where $c(\delta) \to 0$ as $\delta \to 0$.
\end{lem}

\subsection*{Proof of Theorem \ref{densitythm}}
\subsection*{Proof of \eqref{a}.} 
For $r$ and $R$ two positive real parameters, let
\[
\cC(r,R) := \left\{z\in\C\ |z|=e^t\ \mbox{and}\quad r<\psi'(t)<r+R\right\}.
\]
The condition $\cD^+_\varphi(\Gamma)<\frac{1}{p}$ implies that there exists a large positive integer $M$ such that 
\[
n_{\Gamma}(\cC(r,r+pM)) < \frac{pM}{p} = M,\quad r>0.
\]
Hence, for every positive integer $m$, the annulus 
$$\cC_{m,M}:=\cC(pmM+2-\frac{p}{2},pM)$$  contains at most $M-1$ points from $\Gamma$. 

Note that $\cC_{m,M}$ contains $M$ points from $\Sigma_p$. Indeed 
\begin{align*}
    n_{\Sigma_p}(\cC_{m,M})     & = {\rm Card}\left(\left\{n\ge0\ :\ mM\le n\le M(m+1)-1\right\}\right)   = M
\end{align*}
Consequently, there exist $\delta>0$ and $s_m>0$ such that the annulus  
$$ B_m =\cC(s_m,s_m+\delta) \subset \cC_{m,M}$$ 
contains no point from $\Gamma$.

The idea is to complete $\Gamma$ by a sequence $\Gamma'$ (will be chosen later) in such a way that the sequence $\widetilde{\Gamma}=\Gamma\cup \Gamma'$  satisfies conditions of Theorem \ref{thm1}. We denote the points of $\Gamma\cap\cC_{m,M}$ by $\gamma_1^m,\gamma_2^m,\cdots,\gamma_{l_m}^m$, for some $0\leq l_m\le M-1$ (where $|\gamma_j^m|\leq |\gamma_{j+1}^m|$), and write $$\gamma_l^m=\sigma_{mM+l}e^{\delta_{mM+l}}e^{i\theta_{mM+l}},\quad 1\leq l\leq l_m.$$
Recall that each annulus {$\cC_{m,M}$ } contains at least one point from $\Lambda$ for which no thing is associated (free).
 
Take now $N$ a large integer, and consider the grouping of annuli :
$$\ccA_k\ :=\ \bigcup_{m=kN+1}^{kN+N}\ \cC_{m,M},\quad k\geq 0.$$

It follows that, every set $\ccA_k$ contains at least $N$ free points from $\Sigma_p$,   which will assign to $\Gamma'$. We will construct such points in such a way that 
\begin{equation}\label{condi}
 \frac{1}{y_{(kN+2)M}-y_{(kN+1)M}}\left|\sum_{n=(kN+1)M+1}^{(kN+N+1)M}\ \delta_n\right|\ \leq\ C.
\end{equation}
for some absolute constant $C$. Such construction implies that 
\begin{align*}
     \frac{1}{y_{(kN+N+1)M}-y_{(kN+1)M}}\left|\sum_{n=(kN+1)M+1}^{(kN+N+1)M}\ \delta_n \right| \le C  \frac{y_{(kN+2)M}-y_{(kN+1)M}}{y_{(kN+N+1)M}-y_{(kN+1)M}} < \frac{1}{2}
\end{align*}
for large $N$.Thus, from now on, $k$ will be fixed. 

The points of $\Gamma'\cap \ccA_m$ will be chosen within the annuli $B_m=\cC(s_m,s_m+\delta)$. Note that we could even place all missing points in a single annulus $B_m$ and still obtain logarithmically separation sequence. However, in that case the separation 
  constant would depend on $\delta$, $M$, and $N$ and may therefore be rather small. Let us consider all possible sequences 
  $$\displaystyle \Gamma'\subset \bigcup_{m=kN+1}^{kN+N} B_m$$ with separation constants uniformly bounded away from zero, and  write the elements of $\widetilde{\Gamma} = \Gamma\cup \Gamma'$ as $\gamma_n=\sigma_ne^{\delta_n}e^{i\theta_n}$. 
  Note that for any $m$ and $n = mM + 1,\cdots, mM+l_m$ the values $\delta_n$ are already fixed. Moreover, since for these $n$ the corresponding $\sigma_n$ are in the same annulus $\cC_{m,M}$. We have 
\[
\eta(pmM+2-\frac{p}{2}) < y_n <  \eta(pM(m+1)+2-\frac{p}{2})
\]
and 
\[
\eta(pmM+2-\frac{p}{2}) < y_n+\delta_n <  \eta(pM(m+1)+2-\frac{p}{2}).
\]
Hence 
\begin{equation}
    |\delta_n| < \eta(pM(m+1)+2-\frac{p}{2})-\eta(pmM+2-\frac{p}{2})
\end{equation}
whence
\begin{align*}
    \left|\sum_{m=kN+1}^{kN+N} \sum_{n=mM+1}^{mM+l_m} \delta_n\right| & \le \sum_{m=kN+1}^{kN+N} l_m \left(\eta(pM(m+1)+2-\frac{p}{2})-\eta(pmM+2-\frac{p}{2})\right)\\
      & \le M\left(y_{M(kN+N)}-y_{M(kN+1)-1}\right).
\end{align*}
Now, choose all the points of $\Gamma'$ in the annulus $B_{kN+1}$. Then for  \[kNM+jM + l_{kN+j} + 1 < n \le kNM + (j+1)M,\quad 2 \le j\le N-1,\]
we have
\[
   y_{kNM+jM + l_{kN+j} + 1} < y_n \le y_{kNM + (j+1)M}
\]
and the corresponding $\sigma_n\in\cC_{kN+1,M}$ and then
\[
y_{(kN+1)M-1}<\eta(p(kN+1)M+2-\frac{p}{2}) < y_n+\delta_n <  \eta(pM(kN+2)+2-\frac{p}{2})\le y_{M(kN+2)}.
\]
Thus
\[
\delta_n = \delta_n+y_n-y_n\le y_{M(kN+2)}-y_{(kN+j)M+l_{kN+j}+1}\le y_{M(kN+2)}-y_{(kN+j)M+1}.
\]
Recall that we have at least $N$ free point and hence 
\begin{align*}
 \sum_{j=2}^{N-1}\sum_{n=kNM+jM+l_{kN+j}+1}^{kNM+(j+1)M} \delta_n & \le N \sum_{j=2}^{N-1} (y_{M(kN+2)}-y_{(kN+j)M+1})\le N(y_{M(kN+2)}-y_{(kN+N-1)M+1}).
\end{align*}
With this choice, and for for large $N\gg M$. we obtain
\begin{align*}
 \sum_{n=(kN+1)M+1}^{(kN+N+1)M}&\delta_n = \sum_{m=kN+1}^{kN+N}\left(\sum_{n=mM+1}^{mM+l_m}\delta_n+\sum_{n=mM+l_m+1}^{(m+1)M}\delta_n \right)\\
    & \le N(y_{M(kN+2)}-y_{(kN+N-1)M+1}) +  M\left(y_{M(kN+N)}-y_{M(kN+1)-1}\right)<0.
\end{align*}

Now, if instead we choose the terms of $\Gamma'$ in the annulus $B_{kN+N}$, we obtain for every $kNM + jM + l_{kN+j} + 1 < n \le kNM + (j + 1)M$, $0 \le j\le N-3,$ 
that 
\begin{equation}
y_{kNM + jM + l_{kN+j} + 1} < y_n \le y_{kNM + (j + 1)M}.
\end{equation}
Moreover
\begin{equation}
 y_{(kN+N)M-1} < y_n+\delta_n \le y_{M(kN+N+1)} .  
\end{equation}
Hence
\begin{align*}
    \delta_n = \delta_n+y_n-y_n & \ge y_{(kN+N)M-1}-y_{kNM + (j + 1)M}.
\end{align*}
This ensures that
\begin{align*}
    \sum_{n=(kN+1)M+1}^{(kN+N+1)M}\delta_n & = \sum_{m=kN+1}^{kN+N}\left(\sum_{n=mM+1}^{mM+l_m}\delta_n+\sum_{n=mM+l_m+1}^{(m+1)M}\delta_n \right)\\
    & \ge \sum_{m=kN+1}^{kN+N}\sum_{n=mM+l_m+1}^{(m+1)M}\delta_n -M\left(y_{M(kN+N)}-y_{M(kN+1)-1}\right)\\
    & \ge \sum_{j=1}^{N-3}\sum_{n=kNM+jM+l_{kN+j}+1}^{kNM+(j+1)M}\delta_n - O\left(M\left(y_{M(kN+N)}-y_{M(kN+1)-1}\right)\right)\\
    & \ge (N-3)\left(y_{(kN+N)M}-y_{kNM+2M}\right)-O\left(M\left(y_{M(kN+N)}-y_{M(kN+1)-1}\right)\right)>0.
\end{align*}
It  follows that the $\gamma_n$ may be chosen so that
\[
\left|\sum_{n=(kN+1)M+1}^{(kN+N+1)M}\ \delta_n\right| < 1= \frac{y_{(kN+2)M}-y_{(kN+1)M}}{y_{(kN+2)M}-y_{(kN+1)M}}\le (y_{(kN+2)M}-y_{(kN+1)M}).
\]

\subsection*{Proof of \eqref{b}.}
Suppose that $\cD^-_\varphi(\Gamma)>\frac{1}{p}$, then there exists a large positive integer $M$ such that 
\[
n_{\Gamma}(\cC(r,r+pM)) > \frac{pM}{p} = M,\quad r>0.
\]
This ensures that for every integer $m$, the annulus $\cC_{m,M}$ contains at least $M+1$ points from $\Gamma$. Recall that it contains $M$ points from $\Sigma_p.$ Take $N$ such that $N\gg M$ and choose $j_0$ an integer so that $3j_0M< N \le 3(j_0+1)M$. 

For $j_0 \le j \le N-j_0$ and $kNM + jM + 1 \le n\le kNM + (j + 1)M$ we choose in an arbitrary way $\gamma_n\in\Gamma\cap \cC_{j,M}$ and write them as
$\gamma_n=\sigma_ne^{\delta_n}e^{i\theta_n}.$
Then 
\begin{equation}
    |\delta_n| < \eta(pM(j+1)+2-\frac{p}{2})-\eta(pjM+2-\frac{p}{2})
\end{equation}
and 
\begin{align*}
   \left| \sum_{j=j_0}^{N-j_0}\sum_{n=kNM + jM + 1}^{kNM + (j+1)M} \delta_n\right| & \le \sum_{j=j_0}^{N-j_0}M\left(\eta(pM(j+1)+2-\frac{p}{2})-\eta(pjM+2-\frac{p}{2})\right)\\
   & \le M\left(y_{M(N-j_0+1)}-y_{Mj_0-1}\right).
\end{align*}
On the other hand, recall that we still have $N$ free points from the sequence $\Gamma$ in each $\cC_{m,M}$. We now consider two possible choices of $\gamma_n$ for the remaining indices, namely for  
$$kNM + 1 \le n \le kNM + j_0M$$ and  $$kNM + (N-j_0)M + 1 \le n\le kNM + NM.$$
\medskip
\noindent
{\bf $\bullet$ For the first choice.} Assign  points $\gamma_n\in\Gamma\cap\cC_{j,M}$ to $kNM+jM+1\le n \le kNM+(j+1)M$ and $N-j_0 \le j \le N-1$. 
On the other hand, since we have $N/3>j_0M$ free points in $\displaystyle\cup_{j>2N/3}\cC_{j,M}$, we choose $j_0M$ points $\gamma_n$ in $\displaystyle\cup_{j>2N/3}\cC_{j,M}$, for $kNM + 1 \le n\le kNM +j_0M$. Then,
\begin{align*}
    \delta_n & = y_n+\delta_n-y_n \le y_{M(j+1)} -y_{kNM + 1} \le y_{M(N+1)} -y_{kNM + 1}.
\end{align*}
Hence 
\begin{align*}
    \sum_{n=(kN+1)M+1}^{(kN+N+1)M}\delta_n & = \sum_{j=0}^{j_0-1}\sum_{n=kNM+jM+1}^{kNM+(j+1)M}\delta_n + \sum_{j=j_0}^{N-1}\sum_{n=kNM+jM+1}^{kNM+(j+1)M}\delta_n\\
    & \le j_0M(y_{M(N+1)}-y_{kNM+1}) + O\left(M\big(y_{M(N-j_0+1)-y_{Mj_0+1}}\big)\right) <0.
\end{align*}

\medskip
\noindent
$\bullet$ \textbf{Second choice.} Choose $\gamma_n \in \Gamma \cap \cC_{j,M}$ for
\[
kNM + jM + 1 \le n \le kNM + (j + 1)M, \qquad 0 \le j \le j_0 - 1,
\]
and select the remaining $\gamma_n$ so that $\gamma_n \in \bigcup_{j < N/3} \cC_{j,M}$ whenever
\[
kNM + (N - j_0)M + 1 \le n \le kNM + NM.
\]
Arguing as above, we obtain that the corresponding sum of the $\delta_n$ is positive.
This completes the proof of \eqref{b}, as in the proof of \eqref{a}.

\subsection*{Proof of \eqref{c} and \eqref{d}}

It suffices to observe that the lower and upper $\varphi$-densities of a sequence $\Gamma$ can be expressed as
\begin{equation*}
\begin{aligned}
\cD_\varphi^-(\Gamma) &= \liminf_{R\rightarrow\infty}\ \inf_{r\geq 0}\ \frac{n_\Gamma\bigl(\cC(e^{r},e^{r+R})\bigr)}{R \, \psi''(r)},\\[1mm]
\cD_\varphi^+(\Gamma) &= \limsup_{R\rightarrow\infty}\ \sup_{r\geq 0}\ \frac{n_\Gamma\bigl(\cC(e^{r},e^{r+R})\bigr)}{R \, \psi''(r)}.
\end{aligned}
\end{equation*}

The result then follows by applying Lemma \ref{Lemma 5.1} and comparing the density of $\Gamma$ with that of the corresponding reference sequence. 

If $\Gamma$ is an interpolating sequence for $\cF^p_\varphi$, consider the reference sequence $\Sigma_{(1+\varepsilon)\varphi,2}$, which is a sampling sequence for $\cF^p_{(1+\varepsilon)\varphi}$ with density
\[
\cD_\varphi^+(\Sigma_{(1+\varepsilon)\varphi,2}) = \cD_\varphi^-(\Sigma_{(1+\varepsilon)\varphi,2}) = \frac{1+\varepsilon}{p}.
\]
Then, by Lemma \ref{Lemma 5.1} and comparing the densities of interpolating and sampling sequences, we obtain 
\[
\cD_\varphi^+(\Gamma) \le \frac{1}{p}.
\]

Similarly, if $\Gamma$ is a sampling sequence for $\cF^p_\varphi$, consider the reference sequence $\Sigma_{(1-\varepsilon)\varphi,2}$, which is an interpolating sequence for $\cF^p_{(1-\varepsilon)\varphi}$ with density
\[
\cD_\varphi^+(\Sigma_{(1-\varepsilon)\varphi,2}) = \cD_\varphi^-(\Sigma_{(1-\varepsilon)\varphi,2}) = \frac{1-\varepsilon}{p}.
\]
Then, again by Lemma \ref{Lemma 5.1} and comparing densities, we obtain 
\[
\cD_\varphi^-(\Gamma) \ge \frac{1}{p}.
\]
{This completes the proof} of Theorem \ref{densitythm}.

\subsection*{Acknowledgements} The authors would like to thank Sasha Borichev for his helpful remarks and suggestions. The first author  was partially supported by ANR Project ANR-24-CE40-5470.

\end{document}